\documentclass[a4paper,12pt]{article}
\usepackage{amsthm}
\usepackage{amsfonts}
\usepackage{amsmath}
\usepackage{amssymb}
\usepackage{diagrams}
\usepackage{hyperref}
\usepackage[pdftex]{graphicx}

\newtheorem{theorem}{Theorem}[section]
\newtheorem{corollary}[theorem]{Corollary}

\newtheorem{proposition}[theorem]{Proposition}

\theoremstyle{definition}

\newtheorem*{definition*}{Definition}

\theoremstyle{remark}
\newtheorem{remark}[theorem]{Remark}

\numberwithin{equation}{section}

\newcommand{\C}{\mathbb{C}}
\newcommand{\Z}{\mathbb{Z}}

\newcommand{\Q}{\mathbb{Q}}

\newcommand{\diag}{\operatorname{diag}}
\newcommand{\bq}{/\!\!/}

\begin{document}

\begin{center}
\textbf{The classification of compact simply connected biquotients in dimension 6 and 7}

\bigskip

Jason DeVito
\end{center}

\begin{abstract}

We classify all compact simply connected biquotients of dimension $6$ and $7$.  For each $6$-dimensional biquotient, all pairs of groups $(G,H)$ and homomorphisms $H\rightarrow G\times G$ giving rise to it are classified.

\end{abstract}

\section{Introduction}
\label{Intro}

If $M$ is a compact Riemannian manifold, then any subgroup of the isometry group acts on $M$.  When $M$ is homogeneous and the action is free, the quotient, a smooth manifold, is called a biquotient.  Alternatively, biquotients can be defined as quotients of compact Lie groups by two sided actions.  More precisely, given a compact connected Lie group $G$ and a homomorphism $f = (f_1,f_2):H\to G\times G$, there is an induced action of $H$ on $G$ given by $h\ast g = f_1(h)g f_2(h)^{-1}$.  When this action is effectively free, the orbit space, denoted $G\bq H$, naturally has the structure of a smooth manifold and is called a biquotient.

If $G$ is endowed with a bi-invariant metric, then the $H$ action on $G$ is by isometries, and hence induces a metric on the quotient.  By O'Neill's formulas \cite{On1}, this implies that all biquotients carry a metric of non-negative sectional curvature.  Biquotients were introduced by Gromoll and Meyer \cite{GrMe1} when they showed that for a particular embedding of $Sp(1)$ into $Sp(2)\times Sp(2)$, the biquotient $Sp(2)\bq Sp(1)$ is diffeomorphic to an exotic sphere, providing the first example of an exotic sphere with non-negative sectional curvature.  Further, until the recent example due independently to Dearricott \cite{De} and to Grove, Verdiani, and Ziller \cite{GVZ}, all known examples of compact manifolds with positive sectional curvature were diffeomorphic to biquotients.  See \cite{Ber,AW,Wa,Es1,Es2,Baz1}.  Moreover, all known examples of manifolds with almost or quasi-positive curvature are diffeomorphic to biquotients.  See \cite{Wi,PW2,EK,Ke1,Ke2,KT,Ta1,D1,DDRW}.

Biquotients of dimension $6$ were used by Totaro \cite{To2} to construct an infinite family ofnon-negatively curved manifolds with pairwise non-isomorphic rational homotopy types.  Recently, Amann \cite{Am} has used a coarser classification of $7$-dimensional biquotients in the study of $\mathbf{G}_2$ manifolds.

Because each description of a manifold as a biquotient gives rise to a different family of non-negatively curved metrics, it is desirable to not only have a classification of manifolds diffeomorphic to a biquotient, but also to classify which groups give rise to a given manifold.  Totaro \cite{To1} has shown that if $M\cong G\bq H$ is a compact, simply connected biquotient, then $M$ is also diffeomorphic to $G'\bq H'$ where $G'$ is simply connected, $H'$ is connected, and no simple factor of $H'$ acts transitively on any simple factor of $G'$.  We call such biquotients \textit{reduced}, and will classify only the reduced ones.  Further, because we allow our actions to have ineffective kernel, we may replace $H$ by any connected finite cover of itself.  Hence, we may also assume that $H$ is isomorphic to a product of a compact simply connected Lie group and a torus.

Suppose $G_1\bq H_1$ and $G_2\bq H_2$ are both biquotients.  Further, suppose $f:H_2\rightarrow G_1\times G_1$ is a homomorphism defining an action of $H_2$ on $G_1$ which normalizes the $H_1$ action on $G_1$.  Then $f$ gives rise to an action of $H_1\times H_2$ on $G_1\times G_2$ with the $H_2$ factor acting diagonally.  One easily sees that this action is effectively free, and hence, $(G_1\times G_2)\bq (H_1\times H_2)$ is a biquotient.  Noting that this is nothing but the associated bundle to the principal $H_2$-bundle $G_2\rightarrow G_2\bq H_2$, it follows that $(G_1\times G_2)\bq (H_1\times H_2)$ naturally has the structure of a $G_1\bq H_1$ bundle over $G_2\bq H_2$.  We will call such biquotients \textit{decomposable}.

The goal of this paper is to extend our classification of $4$ and $5$ dimensional compact simply connected biquotients \cite{DeV1} to dimension $6$ and $7$.

\begin{theorem}\label{ThmA}Suppose $M^6\cong G\bq H$ is a reduced compact simply connected biquotient.  Then one of the following holds:

a) $G\bq H$ is diffeomorphic to a homogeneous space or Eschenburg's inhomogeneous flag manifold $SU(3)\bq T^2$ \cite{Es1}.

b) $G\bq H$ is decomposable.

c)  $G\bq H$ is diffeomorphic $S^5\times_{T^2} S^3$ or $(S^3)^3\bq T^3$.

\end{theorem}

The manifolds in b) consist of both of the linear $S^4$ bundles over $S^2$, all 3 linear $\C P^2$ bundles over $S^2$, and infinitely many linear $S^2$ bundles with base a $4$-dimensional biquotient $B^4$, i.e., $S^4$, $\mathbb{C}P^2$, $S^2\times S^2$, and $\mathbb{C}P^2\# \pm \mathbb{C}P^2$, see Propositions \ref{calc2}, \ref{cp2s2actions}, and \ref{matlist} and the following discussions.  In particular every such bundle over $B^4$ where the structure group reduces to a circle is a decomposable biquotient, see Propositions \ref{s2buncp2} and \ref{bundoverb4}.  In c), the $T^2$ and $T^3$ actions are linear and there are only finitely many actions which do not give rise to decomposable biquotients, see Propositions \ref{cp2s2actions} and \ref{matlist}.  The manifold $\mathbb{C}P^3\#\mathbb{C}P^3$ arises in case $c)$.

In dimension $7$, we prove

\begin{theorem}\label{ThmB}

Suppose $M^7\cong G\bq H$ is a reduced compact simply connected biquotient.  Then one of the following holds:

a) $G\bq H$ is diffeomorphic to a homogeneous space, an Eschenburg Space $SU(3)\bq S^1$ \cite{Es1}, or the Gromoll-Meyer sphere $Sp(2)\bq Sp(1)$ \cite{GrMe1}.

b) $G\bq H$ is decomposable.

c) $G\bq H$ is diffeomorphic to $S^5\times_{S^1}S^3$, $(SU(3)/SO(3))\times_{S^1}S^3$, or $(S^3)^3\bq T^2$.

\end{theorem}

The manifolds in b) include two nontrivial linear $S^3$ bundles over $S^4$ (both of which have cohomology rings isomorphic to that of $S^3\times S^4$), both linear $S^5$ bundles over $S^2$, and infinitely many $S^3$ bundles over $\mathbb{C}P^2$.  Unlike dimension $6$, in c), there are infinitely many examples which are not decomposable.  The actions on spheres in c) are all linear.

In both dimension $6$ and $7$, we also classify all pairs of groups $(G,H)$ and homomorphisms $H\rightarrow G\times G$ which give rise to reduced compact simply connected biquotients.  See Table \ref{table:gplist} for the list of groups.

Biquotients of the form $(S^3)^3\bq T^3$ are particularly interesting, see Proposition \ref{s3modt3}.  Biquotients of this form fall into three infinite decomposable families and four sporadic examples.  They all have cohomology groups isomorphic to those of $(S^2)^3$, but the ring structure, first Pontryagin class, and second Stiefel-Whitney class distinguish them.

It turns out that in dimensions at most $6$, the cohomology rings and characteristic classes completely determine the diffeomorphism type of a compact simply connected biquotient.  Hence, we can specify for a given $6$ dimensional biquotient $G\bq H$ all the other biquotients which are diffeomorphic to it.  It follows from our classification that, with the exceptions of biquotients diffeomorphic to a bundle over $S^2$ with fiber a $4$-dimensional biquotient, all $6$-dimensional biquotients have at most finitely many descriptions as reduced biquotients.  In dimension $7$, however, the Gromoll-Meyer sphere \cite{GrMe1} and Eschenburg spaces \cite{Es1,Kr,KS1,KS2} show that the cohomology rings and characteristic classes no longer classify the diffeomorphism type of the biquotient.

The outline of the paper is as follows.  In Section 2, we will first cover some preliminary facts about biquotients and their topology, allowing us to classify the possible rational homotopy groups of a biquotient of dimension $6$ or $7$.  In Section 3, we consider each of the possible sequences of rational homotopy groups and, using a theorem of Totaro, find a finite list of pairs of groups $(G,H)$ for which a reduced biquotient $G\bq H$ can have these rational homotopy groups.  In Section 4, we choose several representative pairs $(G,H)$ and, for each pair, classify all effectively free reduced biquotient actions of $H$ on $G$ and, when possible, classify the diffeomorphism types of $G\bq H$.

This paper is a portion of the author's Ph.D. thesis and he is greatly indebted to Wolfgang Ziller for helpful discussions and guidance.  He would also like to thank the referees for many helpful suggestions, including a vast simplification of the proof of Proposition \ref{bundoverb4}.

\section{Preliminaries}
\label{Prelim}

In this section, we review basic facts about biquotients, their rational homotopy theory, and the computation of their cohomology rings and characteristic classes.

\subsection{Background on biquotients}
\label{biq}

A homomorphism $f = (f_1, f_2):H\rightarrow G\times G$, which we will always assume to have finite kernel, defines an action of $H$ on $G$ by $h\ast g = f_1(h)g f_2(h)^{-1}$.

An action is called \textit{effectively free} if whenever any $h\in H$ fixes any point of $G$ then it fixes all points of $G$.  It is called \textit{free} if the only element which fixes any point is the identity. The following proposition is immediate.

\begin{proposition}\label{freetest}
A biquotient action of $H$ on $G$ is effectively free iff whenever $f_1(h)$ is conjugate to $f_2(h)$ in $G$, then $f_1(h) = f_2(h)\in Z(G)$.  Likewise, a biquotient action of $H$ on $G$ is free iff whenever $f_1(h)$ is conjugate to $f_2(h)$ in $G$, then $f_1(h) = f_2(h) = e\in G$.
\end{proposition} 

It follows easily from this that a biquotient action of $H$ on $G$ is (effectively) free iff the action is (effectively) free when restricted to a maximal torus of $H$.  As observed in \cite{Es2}, it follows that rank of $H$ can be at most the rank of $G$.  We will henceforth only consider pairs $(G,H)$ with the rank of $H$ at most that of $G$.

As mentioned in the introduction, when the action of $H$ on $G$ induced by $f$ is effectively free, the quotient $G\bq H$ naturally has the structure of a smooth manifold and is called a biquotient.  If $H = H_1\times H_2$ with each factor embedded into a factor of $G$, the biquotient is sometimes denote $H_1\backslash G/H_2$.  Biquotients were systematically studied in Eschenburg's Habilitation \cite{Es2}.  Also, Totaro \cite{To1} showed that if $M$ is compact, simply connected, and diffeomorphic to a biquotient, then $M$ is diffeomorphic to a biquotient $G\bq H$ where $G$ is compact, simply connected, and semisimple, $H$ is connected, and no simple factor of $H$ acts transitively on any simple factor of $G$.  By definition, a simple factor of $H$ is the projection of a simple factor from the universal cover $\tilde{H}$ to $H$.  Hence, we will always assume our biquotients to have this \textit{reduced} form.  Moreover, since we allow our maps $f:H\rightarrow G\times G$ to have finite kernel, we will assume that $H$ is given as a product of simple Lie groups with a torus.

In order to further reduce the scope of the classification, we will use the following fact:

\begin{proposition}\label{diffeostabilize} Consider the action induced by $f:H\rightarrow G\times G$.  Then, after any of the following modifications of $f$, the new induced action is equivalent to the initial action.

(1)  For any automorphism $f'$ of $H$, replace $f$ with $f\circ f'$

(2)  For any element $g=(g_1,g_2)\in G\times G$, replace $f$ with $C_g\circ f$, where $C_g$ denotes conjugation.

(3)  For any automorphism $f'$ of $G$, replace $f$ with $(f',f')\circ f$.

(4)  If $f':G\times G\rightarrow G\times G$ interchanges the two factors, then replace $f$ with $f'\circ f$.

\end{proposition}

\begin{proof}

For (1), note that a biquotient action is determined by the image of $f$ and that $f$ and $f\circ f'$ have the same image.

For (2), the map $G\rightarrow G$ sending $g$ to $g_1 \, g\, g_2^{-1}$ is an equivariant diffeomorphism.  For (3), the map $f':G\rightarrow G$ is an equivariant diffeomorphism.  Finally, for (4), the inverse map $i:G\rightarrow G$ with $i(g) = g^{-1}$ is an equivariant diffeomorphism.

\end{proof}

We will only classify biquotients and the corresponding actions up to these four modifications.

\

There is a strong link between representation theory and (2) of Proposition \ref{diffeostabilize}, coming from Malcev's Theorem \cite{Ma}:

\begin{theorem}  Suppose $G\in\{SU(n), Sp(n), SO(n)\}$ and let $f,g:H\rightarrow G$ be homomorphisms, thought of as $n$-dimensional complex, quaternionic, or real representations.  If $f$ and $g$ determine equivalent representations, then the images in $G$ are conjugate, except possibly when $G = SO(2n)$.  In this case, the images are always conjugate in $O(2n)$ and conjugate in $SO(2n)$ if at least one irreducible subrepresentation is odd dimensional.  Conversely, if the images of $f$ and $g$ are conjugate, then there is an automorphism $\phi:H\rightarrow H$ for which $f$ and $g\circ \phi$ determine equivalent representations.

\end{theorem}

\subsection{Rational homotopy theory}
\label{rht}

One of the main tools involved in the classification of biquotients is rational homotopy theory.  A simply connected compact manifold $M$ is said to be rationally elliptic if $\dim\pi_\ast(M)_\mathbb{Q} < \infty$ where $\pi_\ast(M)_\Q$ is shorthand for $\pi_\ast(M)\otimes \mathbb{Q}$.  All Lie group are known to be rationally elliptic with all even degree rational homotopy groups trivial.  Further, given any fiber bundle $F\rightarrow E\rightarrow B$, if two of the spaces are rationally elliptic, so is the third by the long exact sequence in rational homotopy groups.  Since any biquotient $G\bq H$ with ineffective kernel $H'$ gives rise to a principal $H/H'$-bundle $H/H'\rightarrow G\rightarrow G\bq H$, it follows that all biquotients are rationally elliptic.

Using rational homotopy theory, Kapovitch and Ziller \cite{KZ} and Totaro \cite{To1} classified all biquotients $G\bq H$ for which $H^\ast(G\bq H;\Q)$ is generated by a single element, i.e., all those biquotients which are rationally homotopy equivalent to a sphere or projective space.  Hence, we will focus our classification on the remaining biquotients.

It turns out the topology of a simply connected rationally elliptic manifold is very constrained.  Let $\pi_{\text{odd}}(M)_\Q$ denote the direct sum of the odd degree rational homotopy groups of $M$, similarly for $\pi_{\text{even}}(M)_\Q$.  Using the notation $|x| = k$ if $x \in \pi_k(M)_\Q$, we have the following theorem, see \cite{FeHaTh}, p. 434:

\begin{theorem}\label{Qrestriction}Let $M^n$ be a simply connected rationally elliptic manifold with $x_i$ a graded basis of $\pi_{odd}(M)_\Q$ and $y_j$ a graded basis of $\pi_{even}(M)_\Q$.

\

\begin{tabular}{ll}

(1)& $\sum |y_j|\leq n$\\

(2)& $n = \sum |x_i| - \sum(|y_j| - 1)$ \\

(3)& $\chi(M) = \sum (-1)^i \dim(H_i(M)_\Q) \geq 0$.\end{tabular}

\end{theorem}

Further, there is a strengthened version of the Hurewicz theorem for rational coefficients, which is proven in \cite{KK}.

\begin{theorem}\label{QHurewicz}  Suppose $X$ is a simply connected topological space with trivial $i$-th rational homotopy group for all $i\leq r$.  Then the Hurewicz map $\pi_k(M)_\Q\rightarrow H_k(M;\Q)$ induces an isomorphism for $k\leq 2r$ and a surjection when $k = 2r+1$.

\end{theorem}

Using these theorems, we prove

\begin{proposition}\label{Qclass}  Let $M$ be a compact simply connected rationally elliptic manifold of dimension $6$ or $7$ which is not rationally equivalent to either a sphere or projective space.  Then the rational homotopy groups of $M$ are abstractly isomorphic to the rational homotopy groups of a product of compact rank one symmetric spaces.  The dimensions of these rational homotopy groups are listed in Table \ref{table:Qclass}.  

\begin{table}[h]

\caption{Possible dimensions of the rational homotopy groups of a rationally elliptic 6- and 7-manifold}\label{table:Qclass}

\begin{center}
\begin{tabular}{|r|r|r|r|r|r|c|}
\hline
$\pi_2$&$\pi_3$&$\pi_4$&$\pi_5$&$\pi_6$&$\pi_7$ &Example\\
\hline
\hline
 &2& & & & &  $S^3\times S^3$\\
\hline
1&1&1& & &1&  $S^2\times S^4$\\
\hline
2&1& &1& & &  $S^2\times \C P^2$\\
\hline
3&3& & & & &  $S^2\times S^2\times S^2$\\
\hline
\hline
 &1&1& & &1&       $S^3\times S^4$\\
\hline
1&1& &1& & &   $S^2\times S^5$ and $\C P^2\times S^3$\\
\hline
2&3& & & & &    $S^2\times S^2\times S^3$\\
\hline

\end{tabular}

\end{center}

\end{table}

\end{proposition}

\begin{proof}

Pavlov \cite{Pa2} proves this in dimension 6, so we focus on dimension 7.

We first show that $M$ cannot be rationally $3$-connected unless it has the same rational homotopy groups as $S^7$.  If $M$ is rationally $3$-connected, that is, $\pi_k(M)_\Q = 0$ for $k\leq 3$, then Theorem \ref{QHurewicz} implies $H_k(M;\Q) =0$ for $k\leq 3$.  Then Poincar\'e duality implies $M$ has the rational cohomology ring of $S^7$.  Again, since $M$ is $3$-connected, Theorem \ref{QHurewicz} implies the map $\pi_k(M)_Q\rightarrow H_k(M;\Q)$ is an isomorphism for $k\leq 6$, so $\pi_k(M)_\Q = 0$ for $k\leq 6$.  By Theorem \ref{Qrestriction}(1), there can be no non-trivial even degree rational homotopy groups.  Then, by Theorem \ref{Qrestriction}(2), $\dim\pi_7(M)_\Q=1$ with all other odd rational homotopy groups vanishing.  Thus, in this case, the rational homotopy groups of $M$ are isomorphic to those of $S^7$.

So, we may assume $M$ is not rationally $3$-connected.  If $M$ is rationally $2$-connected, then by Theorem \ref{QHurewicz}, $\pi_4(M)_\Q$ is isomorphic to $H_4(M;\Q)$.  But, via Poincar\'e duality, $H_4(M;\Q)\cong H_3(M;\Q)$ and $H_3(M;\Q)\cong \pi_3(M)_\Q$ by Theorem \ref{QHurewicz}, so $\dim(\pi_4(M)_\Q) =\dim(\pi_3(M)_\Q)\geq 1$.  By Theorem \ref{Qrestriction}(1), $\dim(\pi_4(M))_\mathbb{Q}\leq 1$, and there can be no other non-trivial even degree rational homotopy groups.  By Theorem \ref{Qrestriction}(2), one has $7 = 3 + \sum_{|x_i|  \geq 5}|x_i| - (4-1)$, so $\dim \pi_7(M)_\mathbb{Q} = 1$, but all other odd degree rational homotopy groups vanish.  In particular, $M$ has the same rational homotopy groups as $S^3\times S^4$.

Hence, we may assume $\pi_2(M)_\Q \neq 0$.  Using a similar analysis as in the previous two cases, one easily sees that $M$ either has the same rational homotopy groups as $(S^2)^2\times S^3$, $S^2\times S^5$ (which are the same as those of $S^3\times \mathbb{C}P^2$), or $\pi_2(M)_\Q \cong \Q$, $\pi_3(M)_\Q \cong \Q^2$, $\pi_4(M)_\Q \cong \Q$, $\pi_5(M)_\Q \cong \Q$, with all other rational homotopy groups are trivial.

However, this last case cannot occur.  For, if $\pi_2(M)_\Q \cong \Q$, $H_2(M;\Z)$ contains a $\Z$-summand.  Let $E$ be the total space of the principal $S^1$-bundle corresponding to a generator of this summand, which is simply connected as shown in \cite{Ko}, and is also rationally elliptic.  The long exact sequence of homotopy groups shows $\pi_k(E)_\Q \cong \pi_k(M)_\Q$ except that $\pi_2(E)_\Q = 0$.  Theorem \ref{QHurewicz} and Poincar\`e duality imply that $H_3(E;\Q) \cong \Q^2 \cong H_5(E;\Q)$ and $H_4(E;\Q) \cong \Q$.  Hence, $\chi(E) <0$, giving a contradiction to Theorem \ref{Qrestriction}(3).
\end{proof}

\subsection{Cohomology ring and characteristic classes}
\label{top}

We now outline techniques, due to Eschenburg \cite{Es3} and Singhof \cite{Si1}, generalizing results of Borel and Hirzebruch \cite{BH1}, for computing the cohomology rings and characteristic classes of biquotients.  

If $G$ is any compact Lie group, we will let $EG$ denote a contractible space on which $G$ acts freely and $BG = EG/G$ will be the classifying space of $G$.  If $f=(f_1,f_2):H\rightarrow G^2$ defines a free biquotient action on $G$, then the principal $H$-bundle $H\rightarrow G\rightarrow G\bq H$ is classified by a map $\phi_H:G\bq H\rightarrow BH$.

Eschenburg \cite{Es3} has shown

\begin{proposition}\label{cohom} Suppose $f:H\rightarrow G\times G$ induces a free biquotient action of $H$ on $G$ and consider the reference fibration $G\rightarrow B G\rightarrow BG\times BG$ induced by the diagonal inclusion $\Delta:G\rightarrow G\times G$.  There is a map $\phi_G:G\bq H\rightarrow BG$ so that the following is, up to homotopy, a pullback of fibrations.

\begin{diagram}
G & \rTo & G\bq H  & \rTo^{\phi_H} & BH \\
 &  & \dTo^{\phi_G} & & \dTo^{Bf} \\
G & \rTo & BG & \rTo^{B\Delta} & BG\times BG\\
\end{diagram}

\end{proposition}

We now fix a coefficient ring $R$ with the property that $H^\ast(G;R) \cong \Lambda_R(x_1, \ldots x_n)$ and $H^\ast(H;R)\cong \Lambda_R(y_1,...y_m)$ are exterior algebras, for example, $R = \mathbb{Q}$, or $R = \Z$ if $H^\ast(G;\Z)$ is torsion free.  Then, it is clear that  $H^\ast(BG;R)\cong R[\overline{x}_1,\ldots,\overline{x}_n]$ where the $\deg(\overline{x}_i) = \deg(x_i)+1$ and $dx_i = \overline{x}_i$ in the spectral sequence associated to the fibration $G\rightarrow EG\rightarrow BG$.   Using this notation, Eschenburg showed

\begin{proposition}\label{differentials}

In the Leray-Serre spectral sequence associated to the reference fibration $G\rightarrow B\Delta G\rightarrow BG\times BG$ in Proposition \ref{cohom}, each $x_i$ is totally transgressive and $dx_i = \overline{x}_i\otimes 1 - 1\otimes \overline{x}_i \in H^\ast(BG;R)\otimes H^\ast(BG;R)\cong H^\ast(BG\times BG;R).$

\end{proposition}

In particular in Proposition \ref{cohom} implies that, using naturality, we can compute the differentials in the fibration $G\rightarrow G\bq H\rightarrow BH$ if we can compute the map $Bf^\ast$ on cohomology.

The method for computing $Bf^\ast$ is due to Borel and Hirzebruch \cite{BH1}.  Fix a maximal torus $T_G\subseteq G$.  The Weyl group of $G$, $W_G$ acts on $T$, and therefore on $H^\ast(T;R)$, and thus, also on $H^\ast(BT;R)$.  We let $H^\ast(BT;R)^W$ denote the Weyl group invariant elements of $H^\ast(BT;R)$.  Then Borel and Hirzebruch \cite{BH1} show

\begin{theorem}\label{toruscomp}

Let $G$ be a compact Lie group with maximal torus $T_G$ and suppose $R$ is a ring with the property that $H^\ast(G;R)$ is an exterior algebra.  Then, the map $i^\ast:H^\ast(BG;R)\rightarrow H^\ast(BT_G;R)$ induced from the inclusion $i:T_G\rightarrow G$ is injective with image $H^\ast(BT_G;R)^W$.

\end{theorem}

By choosing maximal tori $T_H$ and $T_{G\times G}$ for which $f(T_H)\subseteq T_{G\times G}$, this reduces the problem of computing $Bf^\ast:H^\ast(BG\times BG;R)\rightarrow H^\ast(BH;R)$ to the more tractable problem of computing $Bf^\ast :H^\ast(BT_{G\times G}; R)\rightarrow H^\ast(BT_H;R).$

We now describe the computation of $Bf^\ast:H^\ast(BT^n)\rightarrow H^\ast(BT^m)$ given $f:T^m\rightarrow T^n$.  Such an $f$ has the form $f(z_1,...,z_m) = (z_1^{A_{11}} z_2^{A_{12}} \cdots z_m^{A_{1m}}, ..., z_1^{A_{n1}}\cdots z_m^{A_{mn}})$ for some integers $A_{ij}$.  Collect the $A_{ij}$ into a matrix, $A = (A_{ij})$.  Let $x_1,..., x_m\in H^1(T^m)$ denote the canonical generators and similarly for $y_1,...,y_n\in H^1(T^n)$.  Then one easily sees that, with respect to the duals of the $x_i$ and $y_j$, that $f_\ast:H_1(T^m)\rightarrow H_1(T^n)$ is multiplication by the matrix $A$.  Dualizing, $A^t = f^\ast:H^1(T^n)\rightarrow H^1(T^m)$.

We claim that with respect to the elements $\overline{x}_i\in H^2(BT^m)$ and $\overline{y}_j\in H^2(BT^n)$, $Bf^\ast$ is also multiplication by $A^t$.  To see this, we use the following commutative diagram of fibrations.

\begin{diagram}
T^m & \rTo & ET^m & \rTo & BT^m\\
\dTo^{f} & & \dTo & & \dTo^{Bf}\\
T^n & \rTo & ET^n & \rTo & BT^n\\
\end{diagram}

Naturality of the Leray-Serre spectral sequence gives $Bf^\ast(\overline{y_j}) = Bf^\ast(dy_j) = d(f^\ast y_j) = d(A^t y_j)$.  But since $A^t y_j$ is a linear combination of the $y_i$, $d(A^t y_j) = A^t dy_j = A^t \overline{y_j}$.  That is, $Bf^\ast(\overline{y_j}) = A^t \overline{y_j}$.  Finally note that, since $H^\ast (BT^n)$ is generated by $H^2$, this completely determines $Bf^\ast$.

\

Recall that if $H = H_1\times H_2$ with $f:H_1\times H_2\rightarrow G^2$ mapping $H_1$ only into the first factor of $G^2$ and $H_2$ only mapping into the second, we denote the biquotient by $H_1\backslash G/H_2$.  When $H$ has the same rank of $G$, Singhof \cite{Si1} shows the cohomology ring has a particularly nice description.  Note the the map $H_i\rightarrow G$ induces a map $H^\ast(BG;R)\rightarrow H^\ast(BH_i)$ which gives $H^\ast(BH_i)$ the structure of an algebra over $H^\ast(BG;R)$.

\begin{theorem}(Singhof)\label{ringmaxrank}
Suppose $R$ is a ring for which $H^\ast(H_i;R)$ and $H^\ast(G;R)$ are exterior algebras.  For any biquotient of the form $H_1\backslash G/H_2$ with $\operatorname{rk} H_1 + \operatorname{rk} H_2 = \operatorname{rk} G$, the map $\phi_H:H_1\backslash G/ H_2\rightarrow BH_1\times BH_2$ induces an isomorphism $H^\ast(H_1\backslash G/H_2;R)\cong H^\ast(BH_1;R)\otimes_{H^\ast(BG;R)} H^\ast(BH_2;R).$

\end{theorem}

\

For computing Pontryagin classes of the tangent bundle to $G\bq H$, we have the following result due to Singhof \cite{Si1}.

\begin{theorem}\label{pclass} Suppose $f:H\rightarrow G^2$ defines a free biquotient action.  Then the total Pontryagin class of the tangent bundle to $G\bq H$ is given as $$p(G\bq H) = \phi_G^\ast \left[\prod_{\lambda \in \Delta^+G}\left(1+\lambda^2\right)\right] \phi_H^\ast\left[\prod_{\rho\in\Delta^+ H}\left(1+\rho^2\right)^{-1}\right]$$ where $\phi_G$ and $\phi_H$ are given in Proposition \ref{cohom}, and where $\Delta^+ G$ denotes the positive roots of $G$, interpreted as elements of $H^2(BT_G;R)$.

\end{theorem}

The map $\phi_H^\ast$ is computed as the edge homomorphism in the Leray-Serre spectral sequence for the homotopy fibration $G\rightarrow G\bq H\rightarrow BH$ while $\phi_G^\ast$ is computed by noting $\phi_G^\ast \circ B\Delta ^\ast = \phi_H^\ast\circ Bf^\ast$ and using the fact that $B\Delta^\ast$ is surjective.  In fact, as $B\Delta^\ast (\overline{x}\otimes 1) = B\Delta^\ast(1\otimes \overline{x}) = \overline{x}$ for any $\overline{x}\in H^\ast(BG)$, it follows that $\phi_G^\ast$ may be computed as either $\phi_H^\ast \circ Bf_1^\ast$ or $\phi_H^\ast\circ Bf_2^\ast$.  We interpret the roots of $G$ as follows: for $T= T_G\subseteq G$, there is a natural isomorphism between $H^1(T;R)$ and $\operatorname{Hom}(\pi_1(T),R)$.  Further, if $\exp:\mathfrak{t}\rightarrow T$ denotes the exponential map, we can identify $\pi_1(T)$ with $\exp^{-1}(0)$.  As shown in \cite{FH}(Theorem 23.16), the weight lattice of a simply connected Lie group is dual to $\exp^{-1}(0)$.  Hence, we may interpret weights, and in particular, roots of a representation as elements of $H^1(T;R)$.  By using transgressions of generators of $H^1(T;R)$ as generators of $H^2(BT;R)$, we can interpret any weight as an element of $H^2(BT;R)$.

Since we are interested in classifying biquotients of dimension at most $7$, the only possible non-trivial Pontryagin class is $p_1$.  In the notation of Theorem \ref{pclass}, it is easy to see \begin{equation}\label{firstpont} p_1(G\bq H) = \phi_H^\ast\left(\sum_{\lambda \in \Delta^+ G} Bf_i^\ast \lambda^2 - \sum_{\rho\in \Delta^+ H} \rho^2\right)\end{equation} for $i=1$ or $2$.

\

Singhof also proved a similar theorem for Stiefel-Whitney classes.  Here, instead of using the maximal torus and roots, one uses the notion of maximal $2$-groups and $2$-roots of a Lie group.  A $2$-group of a compact Lie group $G$ is any subgroup isomorphic to $(\Z_2)^n$ for some $n$.  We caution that, while the rank of a Lie group is an invariant of its Lie algebra, the $2$-rank of a Lie group, that is, the dimension as a $\Z_2$-vector space of a maximal $2$-group, depends on the group itself.  The $2$-roots of $G$ are defined analogously to the roots:  the adjoint representation of $G$, when restricted to a maximal torus $T_G$, breaks into root spaces.  Likewise when the adjoint representation of $G$ is restricted to a maximal $2$-group $Q_G$ it breaks into $2$-root spaces.  We can view each $2$-root as a map $Q_G\rightarrow \mathbb{Z}_2$ which induces, via the fibration $Q_G\rightarrow EQ_G\rightarrow BQ_G$, a map $H_1(BQ_G; \Z_2)\rightarrow \Z_2$, that is, as an element of $H^1(BQ_G; \Z_2)$.  More generally, given a basis for $Q_G$ as a $\Z_2$-vector space, the dual basis can be canonically identified as generators of $H^1(BQ_G; \Z_2)$.  Using this identification, Singhof has shown \cite{Si1}

\begin{theorem}(Singhof)\label{swclass}

Suppose $f:H\rightarrow G^2$ defines a free biquotient action, then the total Stiefel-Whitney class of the tangent bundle of $G\bq H$ is given as $$w(G\bq H) = \phi_G^\ast\left(\prod_{\lambda \in \Delta^2 G}(1+\lambda)\right)\phi_H^\ast\left(\prod_{\rho\in\Delta^2 H}(1+\rho)\right)^{-1}$$ where $\Delta^2 G$ denotes the $2$-roots of $G$ and $\phi_G^\ast$ and $\phi_H^\ast$ are the maps induced on cohomology with $\Z_2$ coefficients.

\end{theorem}

The $2$-roots of the classical groups and $\mathbf{G}_2$ are recorded in \cite{BH1}.

\subsection{Torus actions on spheres}\label{linearsphere}

We will eventually see that, in many applications, a biquotient of dimension at most $7$ is diffeomorphic to the quotient of a product of odd dimensional spheres by an effectively free linear torus action.  It will often be the case that the particular biquotient action under study is merely effectively free.  In particular, the techniques of Section \ref{top} do not directly apply.  However, we now show the following proposition.

\begin{proposition}\label{torussimp}
Suppose $M\cong \left(S^{2k_1-1}\times \cdots \times S^{2k_m - 1}\right)/T^n$ is diffeomorphic to the quotient of a product of odd dimensional spheres by an effectively free torus action.  Then $M$ is diffeomorphic to biquotient $G\bq H$ where $G = U(k_1)\times \cdots \times U(k_m)$, $H = T^n\times U(k_1-1)\cdots U(k_m-1)$, and $H$ acts on $G$ freely.
\end{proposition}

In particular, Eschenburg's and Singhof's methods can be applied.

\begin{proof}(Proof of Proposition \ref{torussimp})

To begin with, we note that any linear $T^n$ action on $S^{2k-1} = \{(a_1, a_2,..., a_k)\in \mathbb{C}^k: \sum |a_i|^2 = 1\}$ is equivalent to one of the form \begin{equation}\label{lintor} (z_1,..., z_n) \ast (a_1,...,a_k) = (z^{J_1} a_1,..., z^{J_k} a_k). \end{equation}  Here, for $J = (l_1,...,l_n)\in \mathbb{Z}^n$, the notation $z^J$ is shorthand for $z_1^{l_1} \cdots z_n^{l_n}$.

Such an action comes from a biquotient action in the following way.  Consider $H = T^n\times U(k-1)$ with homomorphism $f=(f_1,f_2):H\rightarrow G^2 = U(k)^2$ given by $f_1(z_1,..., z_n, B)  = \diag(z^{J_1},..., z^{J_k})$ and $f_2(z_1,...,z_n, B) = \diag(B,1)$.

Then quotienting out by the $\{e\}\times U(k-1)$ subaction, the torus then acts on $S^{2k-1} = U(k)/U(k-1)$ as in \eqref{lintor}.  Thus, if the $T^n$ action on $S^{2k-1}$ is effectively free, so is the action of $H$ on $G$.

Of course, this same process can be extended to give any linear $T^n$ action on $S^{2k_1-1}\times \cdots\times S^{2k_m-1}$ by setting $H = T^n \times U(k_1-1)\times\cdots \times U(k_m-1)=  T^n\times H'$ and $G = U(k_1)\times \cdots \times U(k_m)$.  Then one sets $$f_1(z_1,...,z_n, B_1,..., B_m) = (\diag(z^{J_{1,1}} ,..., z^{J_{n,1}}), ..., \diag(z^{J_{1,m}},..., z^{J_{n,m}}))$$ and $$f_2(z_1,...,z_n, B_1,..., B_m) = (\diag(B_1, 1), ..., \diag(B_m,1)).$$  Just as before, if the original $T^n$ action on $S^{2k_1-1}\times \cdots \times S^{2k_m-1}$ is effecively free, so is the $H$ action on $G$.

 Because $f_2$ is clearly injective when restricted to $\{e\}\times H'$, we see the kernel of $f$ is a subgroup of $T^n\times \{e\}$.  Because the quotient of a torus is a torus, by dividing out the kernel of $f$, we see that $T^n\times H'/\ker f$ has the form $T^{n'}\times H'$,  for some $n'\leq n$.  So we may assume $f$ is injective.  Because $f_2(H)\cap Z(G) = \{I\}$, this immediately implies, via Proposition \ref{freetest}, that such an action is free iff it is effectively free.

\end{proof}

We now apply Eschenburg's and Singhof's techniques.  Since $T^n$ and $U(k)$ have torsion free cohomology rings, we may choose our coefficient ring to be $R = \mathbb{Z}$.

We now apply Eschenburg's and Singhof's methods in the case where $m=2$.  The case of general $m$ is analogous, but becomes notationally complicated.

Let $T_G = T_{U_{k_1}}\times T_{U_{k_2}}\subseteq G$ denote the standard maximal torus, and similarly for $T_H$.  Let $x_1,..., x_{k_1}, y_1,..., y_{k_2}\in H^1(T_G)$ be the duals of the standard generators of $H_1(T_G)$.  We will use $x$ with no subscript to denote the tuple $(x_1,..., x_{k_1})$, and similarly with all other variables.  As in Section \ref{top}, we may identify $H^2(BT_G)$ with $\mathbb{Z}[\overline{x}, \overline{y}]$ where $dx_i = \overline{x}_i$ in the Serre spectral sequence associated to the fibration $T_G\rightarrow ET_G\rightarrow BT_G$, and similarly for the $y_i$.

The Weyl Group $W_G$ of $G$ is isomorphic to $S_{k_1}\times S_{k_2}$, a product of symmetric groups, and acts by arbitrary permutations of the $\overline{x}_i$ and of the $\overline{y}_j$ separately.   By Theorem \ref{toruscomp}, this implies that we may identify $H^\ast(BG)$ with $H^\ast(BT_G)^{W_G}$, that is, with $$\mathbb{Z}[\sigma_1(\overline{x}), ...., \sigma_{k_1}(\overline{x}), \sigma_1(\overline{y}),..., \sigma_{k_2}(\overline{y})]$$ where $\sigma_i$ denotes the $i$th elementary symmetry polynomial.

In an analogous fashion, we let $u_i\in H^1(T^n)$ denote the duals of the canonical generators of $H_1(T^n)$ and likewise for $s_1,..., s_{k_1 - 1}, t_1, ...., t_{k_2-1}\in H^1(H')$.  Then we identify $H^2(BH)$ with $$\mathbb{Z}[\overline{u}, \sigma_1(\overline{s}),..., \sigma_{k_1-1}(\overline{s}), \sigma_1(\overline{t}),..., \sigma_{k_2-1}(\overline{t})].$$

Then, under the map $f=(f_1,f_2):T^n\times H'\rightarrow G$, one easily sees that, if $A_i$ denotes the matrix of the map $(\pi_i\circ f_1)_\ast: H_1(T^n)\rightarrow H_1(T_{G_i})$ that \begin{center}

\begin{tabular}{llllll}  $Bf_1^\ast(\overline{x})$ & $= \overline{u}A_1^t$ & &   $Bf_1^\ast(\overline{y})$ & $= \overline{u}A_2^t $ & \\ $Bf_2^\ast(\overline{x}_i)$ & $= \overline{s}_i$ & & $Bf_2^\ast(\overline{y}_j)$ & $= \overline{t}_j$ & for $i\neq k_1, j\neq k_2$ \\ $Bf_2^\ast(\overline{x}_{k_1})$ &$= 0$ & & $Bf_2^\ast(\overline{y}_{k_2})$ & $= 0$ & \end{tabular}
\end{center}

It follows that $Bf_2^\ast(\sigma_i(\overline{x})) = \sigma_i(\overline{s})$ for $i\leq k_1-1$ and $Bf_2^\ast (\sigma_{k_1}(\overline{x})) = 0$, and similarly for $Bf_2^\ast(\sigma_i(\overline{y}))$.  In addition, $Bf_1^\ast(\sigma_i(\overline{x})) = \sigma_i(\overline{u} A^t_1)$ and $Bf_1^\ast(\sigma_i(\overline{y})) = \sigma_i(\overline{u} A^t_1)$.

We let $v_i, w_j\in H^\ast(G)$ with $dv_i = \sigma_i(\overline{x})$, $dw_j = \sigma_j(\overline{y})$ in the Leray - Serre spectral sequence associated to the universal fibration $G\rightarrow EG\rightarrow BG$.  So, by Proposition \ref{differentials}, in the spectral sequences associated to $G\rightarrow BG\rightarrow BG\times BG$, we have $dv_i = \sigma_i(\overline{x})\otimes 1 - 1\otimes \sigma_i(\overline{x})$ and similarly for $dw_j$.  Then, in the Serre spectral sequence associated to the homotopy fibration $G\rightarrow G\bq H\rightarrow BH$, we have $dv_i = Bf^\ast(\sigma_i(\overline{x})\otimes 1 - 1\otimes \sigma_i(\overline{x})) = \sigma_i(\overline{u} A^t_1 ) - \sigma_i(\overline{s})$ when $i < k_1$ and $dv_{k_1} = \sigma_{k_1}(\overline{u} A^t_1)$.  Thus, modulo extension problems, we can compute the cohomology groups of $G\bq H$.

In the case where $\operatorname{rk} H = \operatorname{rk} G$, that is, when $m = n$, Theorem \ref{ringmaxrank} shows $\phi_H^\ast$ induces an isomorphism $H^\ast(G\bq H)\cong H^\ast(BT^n)\otimes_{H^\ast(BG)}H^\ast(H')$.  It follows from the description of $dv_i$ and $dw_j$ above that $H^\ast(G\bq H) \cong \mathbb{Z}[\overline{u}]/I$ where $I$ the ideal generated by $\sigma_{k_1}( \overline{u} A_1^t)$ and $\sigma_{k_2}(\overline{u} A_2^t)$.  In general, that is, when $m = n$ but $m$ is arbitrary, an analogous argument gives the following proposition.

\begin{proposition}\label{typcalc}

Suppose $T^n$ acts effectively freely on $S^{2k_1 -1}\times...\times S^{2k_n -1}$ with $A_i$ denoting the matrix describing the inclusion $H_1(T^n)\rightarrow H_1(U(k_i))$.  Then the cohomology ring of the quotient space is isomorphic to $\mathbb{Z}[\overline{u}_1,..., \overline{u}_n]/I$ where $I$ is the ideal generated by $\sigma_{k_i}(\overline{u} A_i^t )$ for $i = 1$ to $n$.

\end{proposition}

We will also be interested in computing the cohomology ring with coefficients in the ring $R = \mathbb{Z}_2$.  To that end, we let $Q_G$ denote a maximal subgroup of $G$ which is isomorphic to a direct sum of copies of $\mathbb{Z}_2$.  According to \cite{BH1}, $Q_G$ is conjugate to a subgroup of the maximal torus $T_G\subseteq G$.  It follows that the map $H^\ast(BG;\mathbb{Z}_2)\rightarrow H^\ast(BQ_G;\mathbb{Z}_2)$ factors as $H^\ast(BG;\mathbb{Z}_2)\rightarrow H^\ast(BT_G;\mathbb{Z}_2)\rightarrow H^\ast(TQ_G; \mathbb{Z}_2)$.  Since $H^\ast(BG)$ is torsion free, the first map in this composition is simply the mod $2$ reduction of the map on the integral level, which we have already computed.  The second map is easily seen to be an isomorphism on the even degree cohomology groups.  We may thus compute $Bf^\ast:H^\ast(BG\times BG;\mathbb{Z}_2)\rightarrow H^\ast(BH;\mathbb{Z}_2)$ as the restriction of the map $Bf^\ast: H^\ast(BQ_G\times BQ_G;\mathbb{Z}_2)\rightarrow H^\ast(BQ_H;\mathbb{Z}_2)$ induced from the inclusion $f:Q_H\rightarrow Q_G\times Q_G$, just as we did to prove Proposition \ref{typcalc}.

\

We now describe the computation of the characteristic classes of $G\bq H$.  For computing the Pontryagin classes, we note the positive roots of $U(k_1)$, interpreted as elements of $H^2(BT_{U(k_1)}) \cong \mathbb{Z}[\overline{x}]$, are the elements of the form $\overline{x}_i - \overline{x}_j$ with $i < j$.  Since all factors of $G$ and $H'$ are isomorphic to $U(k)$ for some $k$, an analogous statements applies to them.  Then, we use equation \eqref{firstpont} to compute $p_1$.

For the Stiefel-Whitney classes, we have the following proposition.

\begin{proposition}\label{easysw} Suppose $T^n$ acts on $S^{2k_1-1}\times\cdots \times  S^{2k_m-1}$ effectively freely.  Then the total Stiefel-Whitney class of the orbit space is  $$w = \phi_G^\ast\left(\prod_{\lambda \in \Delta^+ G}(1+\hat{\lambda})\right) \phi_H^\ast\left( \prod_{\mu\in \Delta^+ H}(1 + \hat{\mu})\right)^{-1}$$ where $\hat{\lambda} \in H^2(BG;\mathbb{Z}_2)$ is the mod $2$ reduction of $\lambda \in H^2(BG; \mathbb{Z})$.

\end{proposition}

\begin{proof}

By using Singhof's formula for the Stiefel-Whitney classes, given in Theorem \ref{swclass}, we see it is enough to show that for each $U(l)$ factor of either $G$ or $H$, that $\prod_{\tau \in \Delta^2 U(l)}(1+\tau) = \prod_{\rho\in \Delta^+ U(l)} (1 + \hat{\rho})$.

To that end, let $\alpha_1,..., \alpha_l$ denote the canonical generators of $H^0(Q_{U(l)};R)$.  We set $d\alpha_i = \overline{\alpha}_i \in H^1(BQ_{U(l)};R)$, where the differential $d$ comes from the Leray-Serre spectral sequence associated to the universal bundle $Q_{U(l)}\rightarrow EQ_{U(l)}\rightarrow BQ_{U(l)}$.  We also recall the elements $\overline{x}_1,\ldots, \overline{x}_l\in H^2(T_{U(l)}; \mathbb{Z})$.

Then, according to \cite{BH1}, the $2$-roots of $U(l)$, interpreted as elements of $H^1(BQ_{U(l)})$, are of the form $\overline{\alpha}_i + \overline{\alpha}_j$, with $i < j$, each with multiplicity $2$.  It follows that, in the formula given in Theorem \ref{swclass}, that $U(l)$ contributes factors of the form $(1 + \overline{\alpha}_i + \overline{\alpha}_j)^2$ and, because we are working mod $2$, this is equal to $ 1 + \overline{\alpha}_i^2 - \overline{\alpha}_j^2$

To finish off the proof, it is enough to note that the reduction $\widehat{\overline{x}}_i$ of $\overline{x}_i$ mod $2$ is $\alpha_i^2$.  This follows by considering the following commutative diagram:

\begin{diagram}
BQ_{U(l)}  & \rTo & BT_{U(l)}\\
\dTo & & \dTo\\
B\mathbb{Z}_2 & \rTo & BS^1 \\
\end{diagram}

The horizontal maps are induced from the natural inclusions, while the vertical maps are induced from the projections to the $i$th factor.  Applying $H^2$ with coefficient ring $R = \mathbb{Z}_2$ to this diagram, the vertical maps become inclusions with image $\alpha_i^2$ and $\widehat{\overline{x}}_i$ respectively.  The result now follows.

\end{proof}

\section{\texorpdfstring{Listing the pairs $(G,H)$}{Listing the pairs (G,H)}}
\label{pairs}

The goal of this section is to prove the following theorem.

\begin{theorem}\label{gplist}  Table \ref{table:gplist} contains all pairs $(G,H)$ of compact Lie groups (with $H$ given only up to finite cover) giving rise to a reduced compact simply connected biquotient $G\bq H$ of dimension $6$ or $7$ whose cohomology ring is not singly generated.  For each such pair, the table characterizes the rational homotopy groups of $G\bq H$ via a prototypical example.

\begin{table}[h!]

\caption{Groups giving rise to a reduced $6$- or $7$-dimensional biquotient}\label{table:gplist}

\begin{center}

\begin{tabular}{|l|l|l|l|}

\hline

$G$ & $H$ & Example & Section\\

\hline

$SU(2)^2$ & $1$ & $S^3\times S^3$ & Corollary \ref{S3class}\\

\hline

$SU(4)\times SU(2) $ & $SU(3)\times SU(2)\times S^1 $ & $S^2\times S^4$ &Section \ref{s2s4ands3s4}\\
$Sp(2)\times SU(2) $ & $SU(2)^2\times S^1 $ & $S^2\times S^4$ & Section \ref{s2s4ands3s4}\\
$Spin(7)\times SU(2)$ & $\mathbf{G}_2\times SU(2)\times S^1$ & $S^2\times S^4$ & Section \ref{s2s4ands3s4}\\
$Spin(8)\times SU(2)$ & $Spin(7)\times SU(2)\times S^1$ & $S^2\times S^4$ & Section \ref{s2s4ands3s4}\\

\hline

$SU(3)$ & $T^2$ & $S^2\times \C P^2$ & \cite{Es2}\\
$SU(3)\times SU(2)$ & $SU(2)\times T^2$ & $S^2\times \C P^2$ & Section \ref{analysiss2cp2}\\
$SU(4)\times SU(2)$ & $Sp(2)\times T^2$ & $S^2\times \C P^2$ & Section \ref{analysiss2cp2}\\

\hline

$SU(2)^3$ & $T^3$ & $S^2\times S^2\times S^2$ & Section \ref{analysiss23}\\

\hline
\hline

$SU(4)\times SU(2)$ & $SU(3)\times SU(2)$ & $S^3\times S^4$ & Section \ref{analysiss3s4} \\
$Sp(2)\times SU(2)$ & $SU(2)^2$ & $S^3\times S^4$ & Section \ref{analysiss3s4}\\
$Spin(7)\times SU(2)$ & $\mathbf{G}_2\times SU(2)$ & $S^3\times S^4$ & Section \ref{analysiss3s4}\\
$Spin(8)\times SU(2)$ & $Spin(7)\times SU(2)$ & $S^3\times S^4$ & Section \ref{analysiss3s4}\\

\hline

$SU(3)$ & $S^1$ & $S^3\times \C P^2$ & \cite{Es1}\\
$SU(3)\times SU(2)$ & $SU(2)\times S^1$ & $S^3\times \C P^2$ & Section \ref{analysiss3cp2}\\
$SU(4)\times SU(2)$ & $Sp(2)\times S^1$ & $S^3\times \C P^2$ & Section \ref{analysiss3cp2}\\

\hline

$SU(2)^3$ & $T^2$ & $S^2\times S^2\times S^3$ & Section \ref{analysiss2s3s3}\\

\hline

\end{tabular}

\end{center}

\end{table}

\end{theorem}

One easily sees that each pair gives rise to a biquotient.  In Section \ref{analysis}, we list all possible effectively free actions.

To prove this theorem, we use a theorem of Totaro's which relates the topology of a reduced biquotient $G\bq H$ to the topology of $G$.  To properly state Totaro's theorem, we need a preliminary definition.  Recall we are assuming $G$ is simply connected, and hence is isomorphic to a product of simple $G_i$.  In particular, $\pi_k(G)_\Q\cong \bigoplus \pi_k(G_i)_\Q$. 

\begin{definition*}  For a biquotient $G\bq H$, let $G_i$ be a simple factor of $G$ and consider the fibration $H\rightarrow G\rightarrow G\bq H$.  We say $G_i$ contributes degree $k$ to $G\bq H$ if the composition $\pi_{2k-1}(H)_\Q\rightarrow \pi_{2k-1}(G)_\Q\rightarrow \pi_{2k-1}(G_i)_\Q$ is not surjective, where the second map is induced from the projection $G\rightarrow G_i$.

\end{definition*}

In particular, since $\pi_{2k}(G)_\Q = 0$ for every Lie group, it follows that if $G_i$ contributes degree $k$ to $G\bq H$, then $\pi_{2k-1}(G\bq H)_\Q\neq 0$.

In order for this homomorphism to fail to be surjective, $\pi_{2k-1}(G_i)_\Q$ must be non-zero.  For each simple group $G$ the values of $k$ for which $\pi_{2k-1}(G)_\Q\neq 0$ are known and tabulated in Table \ref{table:degree}.

\begin{table}[h]

\caption{Degrees of simple Lie groups}\label{table:degree}

\begin{center}

\begin{tabular}{|l|l|}

\hline

$G$ & degrees\\

\hline

$SU(n)$ & 2,3,4,...,n\\

$Spin(2n+1)$ & 2,4,6,...,2n\\

$Sp(n)$ & 2,4,6,...,2n\\

$Spin(2n)$ & 2,4,6,...,2n-2, n\\

$\mathbf{G}_2$ & 2,6\\

$\mathbf{F}_4$ & 2,6,8,12\\

$\mathbf{E}_6$ & 2,5,6,8,9,12\\

$\mathbf{E}_7$ & 2,6,8,10,12,14,18\\

$\mathbf{E}_8$ & 2,8,12,14,18,20,24,30 \\

\hline \end{tabular}

\end{center}

\end{table}

With this, we can now state Totaro's theorem \cite{To1}.

\begin{theorem}\label{Tclass}(Totaro) Suppose $G\bq H$ is a reduced biquotient and $G_i$ is a simple factor of $G$.  Then one of the following occurs.

(1) $G_i$ contributes its maximal degree.

(2) $G_i$ contributes its second highest degree and there is a simple factor $H_i$ of $H$ such that $H_i$ acts on one side of $G_i$ and $G_i/H_i$ is isomorphic to either $SU(2n)/Sp(n)$ for $n\geq 2$, or $Spin(7)/\mathbf{G}_2 = S^7$, $Spin(8)/\mathbf{G}_2 = S^7\times S^7$, or $\mathbf{E}_6/\mathbf{F}_4$.  In each of the four cases, the second highest degree is $2n-1,$ $4$, $4$, or $9$ respectively.

(3) $G_i = Spin(2n)$ with $n\geq 4$, contributing degree $n$ and there is a simple factor $H_i$ of $H$ such that $H_i = Spin(2n-1)$ acts on $G_i$ on one side in the standard way with $G_i/H_i = S^{2n-1}$.

(4) $G_i = SU(2n+1)$ and there is a simple factor $H_i$ of $H$ such that $H_i = SU(2n+1)$ acts on $G_i$ via $h(g) = h \, g \, h^t$.  In this case, $G_i$ contributes degrees $2$, $4$, $6$, ..., $2n$.

\end{theorem}

Note that case (4), which we emphasize is not conjugation, cannot actually occur for a biquotient of dimension at most $7$.  Denoting $G = G_0\times SU(2n+1)$ and $H = H_0 \times SU(2n+1)$, then the projection of the $H$ action on the $SU(2n+1)$ factor is not free.  In fact, $(g,A)\in H_0\times SO(2n+1)\subseteq H_0\times SU(2n+1)$ fixes $I\in SU(2n+1)$.  It follows that the projection of the $H_0\times SO(2n+1)$ action onto $G_0$ must be free.  Since $\dim(G\bq H) = \dim(G_0\bq H_0) \leq 7$, it follows that $\dim(G_0\bq (H_0\times SO(2n+1)))\leq 4$, with equality only in the case of $n=1$.  Finally, the classification of low dimensional biquotients in \cite{DeV1,KZ,Es2} implies that there are no reduced biquotients of dimension at most $4$ where $H$ contains a factor isomorphic to $SO(3)$.  

We point out that the highest non-zero rational homotopy group in Table \ref{table:Qclass} for a product of symmetric spaces is $\pi_{7}$, corresponding to degree $4$.  Hence, Theorem \ref{Tclass} immediately implies that no simple factor of $G$ is an exceptional Lie group.  That is, if $G_i$ is exceptional, the highest degree is greater than $4$, and if $G_i$ does not contribute its highest degree, $G_i = E_6$, occuring in case (2), so contributes degree $9 > 4$.

In the proof of Theorem \ref{gplist}, we will make repeated use of the following proposition.

\begin{proposition}\label{simpfact}Suppose $M = G\bq H$ is a reduced biquotient.  Write $G = G_1\times \ldots \times G_m$, $H = H_1\times \ldots \times H_n\times T^k$ with each $G_i$ and $H_i$ simple.  Then

(1)  Each $G_i$ contributes either its highest degree or second highest degree.  In particular, $m \leq \dim \pi_{\text{odd}} M_\Q$

(2)  $k = \dim \pi_2(M)_\Q$

(3)  $m - n = \dim\pi_3(M)_\Q - \dim\pi_4(M)_\Q$

\end{proposition}

\begin{proof}
We first prove (1)  If $G_i$ does not contribute its highest degree, $G_i$ must fall into class (2), (3), or (4) of Theorem \ref{Tclass}.  We have already ruled out case (4).  Since the highest degree in Table \ref{table:Qclass} is $4$, cases (2) and (3) are very constrained: $(G_1,H_1)$ must be one of $(SU(4),Sp(2))$, $(Spin(7),\mathbf{G}_2)$, $(Spin(8),\mathbf{G}_2)$, or $(Spin(8), Spin(7))$.  In each of these cases, $G_i$ contributes its second highest degree.

To prove (2) and (3), we consider the long exact sequence in rational homotopy groups associated to the fibration $H\rightarrow G\rightarrow G\bq H$.  Recalling that the even rational homotopy groups of a Lie group vanish and that $G$ is simply connected by assumption, we see $\pi_2(M)_Q\cong \pi_1(H)_\Q$ which clearly has dimension $k$.

In addition, a portion of the long exact sequence is $0 \rightarrow \pi_4(M)_\Q\rightarrow \pi_3(H)_\Q\rightarrow \pi_3(G)_\Q\rightarrow \pi_3(M)_\Q\rightarrow 0$.  Since $\dim\pi_3(G)_\Q = m$, and similarly for $H$, (3) follows.

\end{proof}

The next proposition will allows us to reduce the proof of Theorem \ref{gplist} to the case where $M$ is $7$-dimensional, or $M$ has the same rational homotopy groups as $S^3\times S^3$.

\begin{proposition}\label{reduce7to6}

Suppose $H = H'\times S^1$.  Then for any reduced biquotient $G\bq H$, we have $\pi_k(G\bq H) \cong \pi_k(G\bq H')$ for $k\geq 3$, and $\pi_2(G\bq H) \cong \pi_2(G\bq H')\oplus \mathbb{Z}$.

\end{proposition}

\begin{proof}

By first quotienting by the $H'$ action and then the $S^1$ action, we obtain a principal $S^1$-bundle $S^1\rightarrow G\bq H'\rightarrow G\bq H$.  Then, since $\pi_k(S^1) = 0$ for $k\geq 2$ and since $\pi_1(G/H') = 0$, the long exact sequence of homotopy groups associated to this bundle breaks into short exact sequences $0\rightarrow \pi_k(G\bq H')\rightarrow \pi_k(G\bq H)\rightarrow 0$ for $k\geq 3$ and $0\rightarrow \pi_2(G\bq H')\rightarrow \pi_2(G\bq H)\rightarrow \mathbb{Z}\rightarrow 0$.

\end{proof}

To use Proposition \ref{reduce7to6}, we observe that for each $6$-dimensional entry in Table \ref{table:Qclass} with the exception of $S^3\times S^3$, there is a corresponding $7$-dimensional entry for which all rational homotopy groups of degree $3$ or higher agree.  Then Proposition \ref{simpfact} (2) implies that every pair $(G,H)$ with rational homotopy groups isomorphic to those of $S^2\times S^4,$ $S^2\times \mathbb{C}P^2$, or $(S^2)^3$ is of the form $(G,H'\times S^1)$ where $G\bq H'$, by Proposition \ref{reduce7to6} has the same rational homotopy groups as either $S^3\times S^4$, $S^3\times \mathbb{C}P^2$, or $(S^2)^2\times S^3$, respectively.

Recalling the low dimensional isomorphisms $SU(2)\cong Sp(1)\cong Spin(3)$, $Sp(2)\cong Spin(5)$, and $SU(4)\cong Spin(6)$, we now prove Theorem \ref{gplist}.

\begin{proof}(Proof of Theorem \ref{gplist})

By Proposition \ref{reduce7to6}, we need only to handle there case where $M$ has the same rational homotopy groups as $S^3\times S^3$, or $M$ is $7$-dimensional.  We now break into four cases, depending on the rational homotopy groups of $M$.  Recall that $H$ is considered only up to finite connected covers, so we may always assume $H$ is given as a product of a semi-simple group and a torus.

\textbf{Case 1:} $M$ has the same rational homotopy groups of a product of $S^3$s and $S^2$s.

Suppose $M\cong G\bq H$ is a reduced biquotient having the same rational homotopy groups as $(S^2)^m\times (S^3)^n$, that is, $\dim\pi_2(M)_\Q = m$, $\dim\pi_3(M)_\Q = m+n$, and all other rational homotopy groups vanish.  Let $G_i$ be a simple factor of $G$.  By Theorem \ref{Tclass}, $G_i$ contributes at least one degree.  Since the only nontrivial degree is $2$, $G_i$ contributes degree $2$.  In case (2) and (3) of Theorem \ref{Tclass}, degree $2$ does not arise.  Hence, $G_i$ must come from case (1) of Theorem \ref{Tclass}.  In particular, the highest degree of $G_i$ is $2$, so, using Table \ref{table:degree}, we see $G_i$ is isomorphic to $SU(2)$.  Since $2$ is the only degree of $SU(2)$, it follows that there are $n+m$ simple factors of $G$, all isomorphic to $SU(2)$.

By Proposition \ref{simpfact}, the number of circle factors of $H$ is $\dim(\pi_2(M)_\Q) = m$ while the number of simple factors of $H$ is $0$.  Hence, $H\cong T^m$.  Thus, if $M= G\bq H$ has the same rational homotopy groups as $(S^2)^m\times (S^3)^n$, then $(G,H) = ((SU(2))^{m+n}, T^m)$.  This proves Theorem \ref{gplist} in the first case and last case in dimension $6$, as well as the last case in dimension $7$.

\

\textbf{Case 2:} $M$ has the same rational homotopy groups as $S^3\times S^4$.

Suppose $M\cong G\bq H$ is a reduced biquotient having the same rational homotopy groups as $S^3\times S^4$.  Using Proposition \ref{simpfact}, we see that since $\dim(\pi_{odd}(M)_\Q) = 2$, $G$ has at most two factors and since $\dim\pi_3(M)_\Q = \dim\pi_4(M)_\Q$, $H$ has the same number of simple factors as $G$.  Finally, since $\dim \pi_2(M)_\Q = 0$, $H$ has no circle factors.

We now break into cases depending on the number of simple factors of $G$.

\textbf{Case 2a} $G$ is simple.

Suppose $G$ is simple, and therefore, that $H$ is also simple.  Since $G$ must contribute degree $4$, this is either the highest degree of $G$, or $G$ falls into case (2) or (3) of Theorem \ref{Tclass}.  In case (2) or (3), we find that $G/H$ is either $S^7$ or $S^7\times S^7$, neither of which have the correct rational homotopy groups.  So, $G$ falls into case (1) of Theorem \ref{Tclass}, implying that $4$ is the highest degree of $G$.  Thus, $G = SU(4) = Spin(6)$ or $G = Sp(2) = Spin(5)$.  If $G = SU(4)$, then $\dim H = \dim G - 7 = 8$.  Since $H$ is simple, $H = SU(3)$, giving the pair $(SU(4), SU(3))$.  If $G = Sp(2)$, then $\dim H = \dim G - 7 = 3$, so $H = SU(2) = Sp(1)$, giving the pair $(Sp(2), Sp(1))$.  However, in \cite{KZ}, Kapovitch and Ziller show that for the pairs $(SU(4), SU(3))$ and $(Sp(2), Sp(1))$, all biquotients are rationally $S^7$, a contradiction.  Hence, $G$ can not be simple.

\

\textbf{Case 2b} $G = G_1\times G_2$ has two simple factors.

Suppose $G = G_1\times G_2$ and $H = H_1\times H_2$ are both products of two simple factors.  Then one factor of $G$, say $G_1$, must contribute degree $4$ to $M$ while the other factor, $G_2$ must contribute degree $2$.  Since degree $2$ does not appear in case (2) or (3) of Theorem \ref{Tclass}, $G_2$ must contribute its highest degree of $2$, so $G_2= SU(2)$.  Then $\dim H_1 + \dim H_2 = \dim G_1 + \dim G_2 - 7 = \dim G_1 - 4$.

The degree $4$ which $G_1$ contributes is either the highest degree of $G_1$, so $G_1 = SU(4)$ or $Sp(2)$, or $G_1$ comes from cases (2) and (3) of Theorem \ref{Tclass}.  In the first case, from the classification of simple Lie groups, it follows that $H$ is, up to cover, isomorphic to either $SU(3)\times SU(2)$ or $SU(2)\times SU(2)$, respectively, giving rise to two of the entries of Table \ref{table:gplist}.

So, we may assume that $G_1$ contributes a degree as in case (2) or (3) of Theorem \ref{Tclass}, so $(G_1,H_1) = (Spin(7), \mathbf{G}_2), (Spin(8), \mathbf{G}_2),$ or $(Spin(8), Spin(7))$.  Since $\dim H_2 = \dim G_1  - 4 - \dim H_1,$ we find $\dim H_2 = 3$, $10$, or $3$ respectively.  Since $H_2$ is simple, one easily sees that this implies that $H$ is isomorphic to $SU(2)$, $Sp(2)$, or $SU(2)$ respectively.

However, the case $(G,H) = (Spin(8)\times SU(2), \mathbf{G}_2\times Sp(2))$ can not give rise to a biquotient.  To see this, note that there are no nontrivial homomorphisms from $H$ into $SU(2)$, so the projection of the $H$ action to the $SU(2)$ factor of $G$ is trivial.  It follows that the projection of the $H$ action to the $Spin(8)$ factor of $G$ must be effectively free.  Then $Spin(8)\bq H$ is a reduced $4$-dimensional biquotient.  But these have already been classified \cite{D1,KZ,Es2} and, in particular, there is no reduced biquotient of the form $Spin(8)\bq H$.  This completes case 2b, and hence also case 2.

\

\textbf{Case 3} $M$ has the same rational homotopy groups as $S^3\times \mathbb{C}P^2$.

Suppose $M\cong G\bq H$ is a reduced biquotient having the same rational homotopy groups as $S^3\times \mathbb{C}P^2$.  By Proposition \ref{simpfact}, $G$ contains at most two factors, $H$ has one fewer simple factor, and $H$ contains an $S^1$ factor.

We now break into cases depending on the number of simple factors of $G$.

\textbf{Case 3a}: $G$ is simple.

If $G$ is simple, then $H$ has no simple factors, so $H\cong S^1$.  Then $\dim G = 7 + \dim S^1 = 8$.  From the classification of simple Lie groups, it follows that $G \cong SU(3)$.

\textbf{Case 3b}: $G = G_1\times G_2$ has two simple factors.

Suppose $G = G_1\times G_2$ and $H = H'\times S^1$.  We assume $G_1$ contributes degree $3$ while $G_2$ contributes degree $2$.  As in the previous case, this implies $G_2 = SU(2)$.

Since $G_1$ contributes degree $3$, only cases (1) and (2) of Theorem \ref{Tclass} can occur, and, if case (2) occurs, then $G_1 = SU(4)$ with $H' = Sp(2)$, giving rise to the pair $(SU(4)\times SU(2), Sp(2)\times S^1)$.

So, we assume $G_1$ contributes its highest degree of $3$, which implies $G\cong SU(3)$.  Then $\dim H'= 3$, so $H'$ is isomorphic to $SU(2)$.  This gives the entry $(G,H) = (SU(3)\times SU(2), SU(2)\times S^1)$, completing Table \ref{table:gplist}.

This completes the proof of case 3, and hence, of Theorem \ref{gplist}.

\end{proof}

In case 1 of the proof above, we see that if $M = G\bq H$ is a reduced biquotient having the same rational homotopy groups as $(S^3)^n$, then $(G,H) = (SU(2)^n, \{e\})$.  Thus, we have the following corollary.

\begin{corollary}\label{S3class}

If $M$ is a compact simply connected biquotient with $ n = \dim\pi_3(M)_\mathbb{Q}$ and if all other rational homotopy groups vanish, then $M$ is diffeomorphic to $(S^3)^n$.

\end{corollary}

\section{\texorpdfstring{Analyzing some pairs $(G,H)$}{Analyzing some pairs (G,H)}}
\label{analysis}

In this section, we will select several pairs $(G,H)$ from Table \ref{table:gplist} and classify all the effectively free biquotient actions of $H$ on $G$ and then, when possible, determine the diffeomorphism type of the quotient.  Corollary \ref{S3class} and the preceding discussion handle this for the first case in Table \ref{table:gplist}, when $M$ has rational homotopy groups isomorphic to those of $S^3\times S^3$.

We also note that in many examples, $G$ contains a factor isomorphic to $SU(2)$.  Thus, the following proposition will be used repeatedly.

\begin{proposition}\label{su2biquotients}

Suppose $G = G_1\times SU(2)$ and $H = H'\times T^n$ is product of a semisimple and simply connected compact group $H'$ with a torus.  Suppose $f:H\rightarrow G\times G$ defines an effectively free reduced biquotient action and assume the projection of the $H$ action to the $SU(2)$ factor of $G$ is nontrivial.  Then, exactly one of the following occurs.

(1)  The torus factor $T^n$ acts non-trivially on the $SU(2)$ factor of $G$ while $H'$ acts trivially.

(2)  $T^n$ acts trivially on the $SU(2)$ factor of $G$ and at most one simple factor of $H'$, isomorphic to $SU(2)$, acts by conjugation with all other simple factors of $H'$ acting trivially.  Further, the projection of the $H$ action to the first factor of $G$ must be effectively free.

\end{proposition}

\begin{proof}

Of course, if the projection of the $H'$ action to the $SU(2)$ factor of $G$ is trivial, $T^n$ must act non-trivially, so we may assume $H'$ acts non-trivially.  In particular there is a simple factor $H_1$ of $H'$ which acts non-trivially.  Suppose this action is defined by a homomorphism $f = (f_1,f_2):H_1\rightarrow SU(2)^2$.

From the classification of simple Lie groups, it follows that every simply connected simple Lie group which is not isomorphic to $SU(2)$ has dimension greater than three.  In particular, if $H_1$ is not isomorphic to $SU(2$),  both $f_i$ homomorphisms have positive dimensional kernel.  Since $H_1$ is simple, this implies $f$ is trivial, contradicting the fact that $H_1$ acts non-trivially on $SU(2)$.  Thus, $H_1\cong SU(2)$.

Now, up to conjugation, there are precisely two homomorphisms $SU(2)\rightarrow SU(2)$, the trivial homomorphism and the identity.  It follows that the only $SU(2)$ biquotient actions on itself are, up to equivalence, trivial, left multiplication, and conjugation.  However, since the action is reduced, left multiplication cannot occur.  Hence, the non-trivial actio of $H_1 = SU(2)$ on $SU(2)$ is conjugation.  In particular, $f$ is the diagonal embedding $SU(2)\rightarrow SU(2)^2$ given by $A\mapsto (A,A)$.  The image is a maximal connected subgroup of $SU(2)\times SU(2)$.  It follows that $f$, when restricted to any other simple factor of $H'$ or to $T^n$, must be trivial.

Finally, The projection of the $H$ action to the $SU(2)$ factor of $G$ fixes the identity, and hence, the projection of the $H$ action to $G_1$ must be effectively free.

\end{proof}

\subsection{\texorpdfstring{Biquotients with $\pi_\ast(G\bq H)_\Q\cong \pi_\ast(S^2\times S^4)_\Q$}{Rational homotopy groups like S2 x S4 }}
\label{s2s4ands3s4}

In this section, we classify many biquotients having the same rational homotopy groups as $S^2\times S^4$.

Let $(G,H) = (G_1\times SU(2), H_1\times SU(2)\times S^1)$ denote one of the four entries in Table \ref{table:gplist} having the same rational homotopy groups as $S^2\times S^4$.  That is,$$ (G_1,H_1)\in \{(SU(4), SU(3)), (Sp(2), SU(2)), (Spin(7), \mathbf{G}_2), (Spin(8), Spin(7)).$$

By Proposition \ref{reduce7to6}, we see that the classification of effectively free biquotients with rational homotopy groups isomorphic to $S^3\times S^4$ will aid in the classification of those with rational homotopy groups matching those of $S^2\times S^4$, so we begin with this case.

\begin{proposition}\label{S3S4class}  Suppose $f:H_1\times SU(2)\rightarrow (G_1\times SU(2))^2$ gives rise to a reduced effectively free biquotient action.  Then, up to equivalence, the $H_1$ factor acts on only one side of $G_1$ with $G_1/H_1 = S^7$ or $G_1/H_1 = Sp(2)/\Delta Sp(1) \cong T^1 S^4$.  Further, $H_2= SU(2)$ acts on $G_1/H_1$ freely with quotient $S^4$ and acts on $SU(2)$ either trivially or by conjugation.

\end{proposition}

When $G_1/H_1 = S^7$, the $SU(2)$ action is the Hopf action.  In the other case, when $G_1/H_1 = Sp(2)/\Delta Sp(1)$, the $SU(2) = Sp(1)$ action is by left multiplication by $\diag(p,1) \subseteq Sp(2)$ with $p\in Sp(1)$.  We note that the double cover $Sp(2)\rightarrow SO(5)$ induces a diffeomorphism $Sp(2)/\Delta Sp(1)\cong SO(5)/SO(3)\cong T^1 S^4$.

We now prove Proposition \ref{S3S4class}.

\begin{proof}  (Proof of Proposition \ref{S3S4class})

By Proposition \ref{su2biquotients}, we may assume without loss of generality that the projection of the $H_1$ action onto the $SU(2)$ factor of $G$ is trivial and that the action by $H_2= SU(2)$ is either trivial or by conjugation.  In either case, the projection of the full $H$ action on the $SU(2)$ factor fixes the identity, so the projection of the $H$ action on the $G_1$ factor must be effectively free.

Then $G_1\bq H$ is a $4$-dimensional biquotient.  These were classified in \cite{KZ}:  $H_1$ acts only one one side of $G_1$ and either $G_1/H_1 = S^7$, which $SU(2)$ then acts on via the Hopf action, or $G_1/H_1 = Sp(2)/ \Delta Sp(1)$ and $Sp(1)$ acts on $G_1/H_1$ by left multiplication by $\diag(p,1)$.

\end{proof}

Using Proposition \ref{S3S4class}, we may easily classify all actions giving rise to biquotients having the rational homotopy groups of $S^2\times S^4$.

\begin{proposition}\label{S2S4class}  Suppose $f:H = H_1\times SU(2)\times S^1\rightarrow \left(G_1\times SU(2)\right)^2$ gives rise to a reduced effectively free biquotient action.  Then the projection of the $H_1\times SU(2)$ subaction onto $G_1$ is effectively free with quotient $S^4$ while the projection to the $SU(2)$ factor of $G$ is trivial.  The circle factor of $H$ acts linearly on $S^4$ and, up to ineffective kernel, as the Hopf action on the $SU(2)$ factor of $G$.

\end{proposition}

In particular, such biquotients are always decomposable, being the total space of a linear $S^4$ bundle over $S^2$.  But linear $S^m$ bundles over $S^2$ are classified by $[S^2, BO(m+1)]\cong \pi_1(O(m+1))\cong \mathbb{Z}_2$ if $m\geq 2$.  Thus, there are precisely two linear $S^m$ bundles over $S^2$ for any $m\geq 2$.  Using Poincar\'e duality and the Gysin sequence associated to $S^m\rightarrow E\rightarrow S^2$, one easily sees that for $m\geq 3$, the cohomology ring of $E$ is isomorphic to that of a product.  However, as shown in \cite{Ge} (Lemma 8.2.5), the second Stiefel-Whitney class, a homotopy invariant \cite{Wu}, distinguishes them.  That is, the total spaces of the two linear $S^m$-bundles over $S^2$ are not even homotopy equivalent.

We now prove Proposition \ref{S2S4class}.

\begin{proof}(Proof of Proposition \ref{S2S4class})

By Proposition \ref{reduce7to6}, the restriction of the $H = H_1\times SU(2)\times S^1$ action to $H_1\times SU(2)\times \{e\}$ defines a biquotient with rational homotopy group isomorphic to those of $S^3\times S^4$, so Proposition \ref{S3S4class} classifies these actions.  But, if the $SU(2)$ factor of $H$ acts by conjugation on the $SU(2)$ factor of $G$, then, by Proposition \ref{su2biquotients}, the circle factor of $H$ must acts trivially on $SU(2)$.  It follows that in this case, the projection of the $H_1\times SU(2)\times S^1$ action to the $G_1$ factor of $G$ must be effectively free.  But the rank of $H_1\times SU(2)\times S^1$ is bigger than that of $G_1$ in every case, so this can not occur.  It follows that the $SU(2)$ factor of $H$ must act trivially on the $SU(2)$ factor of $G$.  Thus, again using Proposition \ref{S3S4class}, $H_1\times SU(2)$ acts freely on $G_1$ with $G_1\bq (H_1\times SU(2))\cong S^4$.

If we equip $G_1\times SU(2)$ with a bi-invariant metric, the induced metric on $(G_1\bq (H_1\times SU(2)))\times SU(2) \cong S^4\times S^3$ is a product of round metrics and the induced $S^1$ action is by isometries, hence linear.  But, every linear action of $S^1$ on $S^4$ has a fixed point, so $S^1$ must act, up to ineffective kernel, as the Hopf map on $S^3$.

\end{proof}

\

We now state the classification results for each of the pairs $(G,H)$ in Table \ref{table:gplist} which have rational homotopy groups isomorphic to those of $S^2\times S^4$.  We use the notation $S^4\, \hat{\times}\, S^2$ to denote the unique non-trivial $S^4$ bundle over $S^2$, and we use the notation $R(\theta)$ to denote the standard $2\times 2$ rotation matrix, $R(\theta) = \begin{bmatrix} \cos\theta &-\sin \theta\\ \sin\theta & \cos\theta\end{bmatrix}$.

\begin{theorem}\label{su4su21}
For $(G,H) = (SU(4)\times SU(2), SU(3)\times SU(2)\times S^1)$, every effectively free biquotient action is equivalent to one defined by $f:H\rightarrow G^2$ with $$f(A,B,z) = \left( \left(\diag(z^m A, \overline{z}^{3m}), \diag(z^l, \overline{z}^l)\right), \left(\diag(z^n B, \overline{z}^n B), I\right) \right)$$ where $\gcd(l,m,n) = 1$ and $l|2\gcd(3m,n)$.  If $l$ is even and both $m$ and $n$ are odd, the quotient is diffeomorphic to $S^4\,\hat{\times}\, S^2$.  Otherwise, it is diffeomorphic to $S^4\times S^2$.

\end{theorem}

\begin{theorem}\label{sp2su21}

For $(G,H) = (Sp(2)\times SU(2), SU(2)^2\times S^1)$ every effectively free biquotient action is equivalent to one defined by one of two families of homomorphisms $f:H\rightarrow G^2$:  first,  $$f(p,q,z) = \left( \left(\diag(p,q), \diag(z^l, \overline{z}^l)\right), \left(\diag(z^m, z^n), I\right)\right)$$ and second, with $z = e^{i\theta}$, $$f(p,q,z) = \left( \left( R(m\theta) \cdot \diag(p,p), \diag(z^l,\overline{z}^l)\right), \left(\diag(q,z^n), I\right)\right).$$  In both cases, $\gcd(l,m,n) = 1$ and $l=1$ or $l=2$ and $m$ and $n$ are not both even.  In addition, the quotient is diffeomorphic to $S^4\,\hat{\times}\, S^2$ when $l=2$ and both $m$ and $n$ are odd.  Otherwise, it is diffeomorphic to $S^4\times S^2$.

\end{theorem}

\begin{theorem}\label{spin8su21}

For $(G,H) = (Spin(8)\times SU(2), Spin(7)\times SU(2)\times S^1)$, every effectively free biquotient action is equivalent to one given by $$f(A,B,z) = \left( \left(\diag(A,1), \diag(z^l, \overline{z}^l) \right), \left( \diag( z^mB, z^n B), I\right) \right)$$ where the notation $(z^m  B, z^n B)\subseteq Spin(8)$ is the lift of the block diagonal embedding $U(2)\subseteq \Delta SO(4)\subseteq SO(4)^2\subseteq SO(8)$.  In addition, $\gcd(l,m,n) = 1$ and either $l = 1$ or $l=2$ and $m$ and $n$ are not both even.  If $l=2$ and both $m$ and $n$ are odd, the quotient is diffeomorphic to $S^4\, \hat{\times}\, S^2$.  Otherwise, it is diffeomorphic to $S^4\times S^2$.

\end{theorem}

For the next theorem, we recall that there is, up to conjugacy, a unique non-trivial homomorphism $\mathbf{G}_2\rightarrow SO(7)$.  We let $\pi:SU(2)\rightarrow SO(3)$ denote the double covering map.

\begin{theorem}\label{spin7su21}

For $(G,H) = (Spin(7)\times SU(2), \mathbf{G}_2\times SU(2)\times S^1)$, every effectively free biquotient action is equivalent to one defined by the lift of $f:G_2\times SU(2)\times S^1\rightarrow (SO(7)\times SU(2))^2$ where, with $z=e^{i\theta}$, $$f(A,B,z) = \left( \left(A, \diag(z^l,\overline{z}^l) \right), \left(\diag(\pi(B),R(m\theta), R(n\theta)) , I\right)\right).$$   In addition, $l = \gcd(l,m,n) \in \{1,2\}$ and $m$ and $n$ have the same parity.  When $l = \gcd(l,m,n)= 2$, then $m/2$ and $n/2$ have different parities and the quotient is diffeomorphic to $S^4\,\hat{\times}\, S^2$.  Otherwise, it is diffeomorphic to $S^4\times S^2$.
\end{theorem}

We will only prove Theorem \ref{sp2su21} in the first case and Theorem \ref{spin7su21}, the other proofs being similar.  In each case, we determine the diffeomorphism type of the quotient by computing the second Stiefel-Whitney class: for the first case, we will use prior results on $5$-dimensional biquotients, while in the second, we will rely on the techniques from Section \ref{top}, particularly Theorem \ref{swclass}.

\subsubsection{Proof of part 1 of Theorem \ref{sp2su21}}\label{calc4}

We begin with $(G,H) = (Sp(2)\times Sp(1), Sp(1)^2\times S^1)$ where the action of $Sp(1)^2$ on $Sp(2)$ is given by $(p,q)\ast A  = \diag(p,q) A $.  Using Proposition \ref{diffeostabilize}, we may assume these biquotients are defined by a map $f = (f_1,f_2):H\rightarrow G\times G$ where $$f_1(p,q,z) = \left(\begin{bmatrix} p & \\ & q\end{bmatrix}, 1\right) \text{ and }f_2(p,q,z) = \left(\begin{bmatrix}z^m & \\ & z^n\end{bmatrix}, z^l\right)$$  where we assume that $\gcd(l,m,n) = 1$. Further, the map $S^1\rightarrow S^1$ with $z\mapsto \overline{z}$ is an automorphism, so we may assume $l\geq 1$.

Suppose $z$ is an $l$-th root of $1$.  Then, $f_1(z^m, z^n, z)$ is equal to $f_2(z^m, z^n, z)$ and so, by Proposition \ref{freetest}, $f_1(z^m, z^n, z) = f_2(z^m, z^n, z) \in Z(G)$.  Because $Z(Sp(2)) =\{\pm I\}$, this implies that $z^m = z^n \in \{\pm 1\}$.  In particular, every $l$-th root of $1$ must be a $2m$-th and $2n$-th root of $1$, so $l$ divides $2(m,n)$.  But $\gcd(l,m,n) = 1$, so $l = 1$ or $2$.  If $l=2$ and $z=-1$, then the equation $z^m = z^n $ forces $m$ and $n$ to have the same parity, so $m$ and $n$ must be odd.

At this point, one can immediately see that when $l=1$, the quotient is diffeomorphic to $S^4\times S^2$, independent of $m$ and $n$.  For, according to \cite{BH1}, $H$ has no nontrivial $2$-roots and the $2$-roots of $G$ all have multiplicity $4$, so Theorem \ref{swclass} thus implies the smallest possible nontrivial Stiefel-Whitney class is in dimension $4$.

\begin{proposition}\label{calc2}

When $l=2$ and $m$ and $n$ are both odd, the quotient $G\bq H$ is diffeomorphic to $S^4\,\hat{\times} \, S^2$.

\end{proposition}

\begin{proof}

Consider the double cover $\pi:Sp(2)\rightarrow SO(5)$, induced by taking the second exterior power of the standard representation of $Sp(2)$ on $\C^4$.  Recall that $\pi$ maps $Sp(1)^2\subseteq Sp(2)$ onto $SO(4)\subseteq SO(5)$ and it maps $\diag(z^m,z^n)\subseteq Sp(2)$ onto $\diag(R((m+n)\theta), R((m-n)\theta),1)\subseteq SO(5)$, where $R(\theta)$ denotes the standard $2\times 2$ rotation matrix.  Then, $\pi$ induces a diffeomorphism $\overline{\pi}:Sp(2)\times Sp(1)/Sp(1)^2\times S^1 \rightarrow SO(5)\times Sp(1)/SO(4)\times S^1$ where the $S^1$ action on $SO(5)/SO(4) = S^4$ is given by multiplication by $\diag(R((m+n)\theta), R((m-n)\theta),1)$.

This $S^1$ action on $S^4$ preserves the equatorial $S^3\subseteq S^4$, and hence we see that $M^5 = S^3\times_{S^1} S^3$ is naturally a codimension 1 submanifold of $ N^6 = S^4 \times_{S^1} S^3$.  Let $i:M\rightarrow N$ denote the inclusion map.

Then we have $i^\ast(TN) = TM\oplus \nu$ where $\nu$ is the rank $1$ normal bundle.  It follows that $i^\ast(w_2(TN)) = w_2(i^\ast TN) = w_2(TM\oplus \nu) = w_2(TM)$ since $w_1(TM) = w_2(\nu) = 0$.  But, it is shown in \cite{DeV1} that $w_2(TM)\neq 0$.  It follows that $w_2(TN)\neq 0$ as well, so $N^6$ must be diffeomorphic to $S^4\, \hat{\times} S^2$.  

\end{proof}

\subsubsection{Proof of Theorem \ref{spin7su21}}

We now prove Theorem \ref{spin7su21}.  The main idea is to relate actions of $\tilde{H} = \mathbf{G}_2\times SU(2)\times S^1$ on $\tilde{G} = Spin(7)\times SU(2)$ to actions of $H = \mathbf{G}_2\times SO(3)\times S^1$ on $G = SO(7)\times SU(2)$.  We let $\pi:Spin(7)\times SU(2)\rightarrow SO(7)\times SU(2)$ denote the double cover.

To begin with, recall there is, up to conjugation, a unique non-trivial homomorphism $g_1:\mathbf{G}_2\rightarrow SO(7)$.  Let $g_2:SO(3)\rightarrow SO(7)$ denote the usual inclusion as a $3\times 3$ block.  The composition $\mathbf{G}_2\times SU(2)\rightarrow \mathbf{G}_2\times SO(3)\xrightarrow{g_1\times g_2} SO(7)^2$ admits a lift $\tilde{g}:\mathbf{G}_2\times SU(2)\rightarrow Spin(7)^2$.  Then, it is shown in \cite{Es2} that the only effectively free biquotient action of $\mathbf{G}_2\times SU(2)$ on $Spin(7)$ is induced by $\tilde{g}$.  In fact, we have $\mathbf{G}_2\backslash Spin(7) = S^7$ which $SU(2)$ then acts on via the Hopf action.  In particular, the lift of $SU(2)\rightarrow SO(3)\xrightarrow{g_2} SO(7)$ to $Spin(7)$ is injective.  We use this lift to identify $SU(2)$ as a subgroup of $Spin(7)$.

Consider any homomorphism $f:\tilde{H}\rightarrow \tilde{G}^2$ with $f|_{\mathbf{G}_2\times SU(2)\times\{I\}} = i\circ\tilde{g}$ with $i:Spin(7)\rightarrow Spin(7)\times SU(2)$ denoting the natural inclusion.  Then $\pi \circ f$ descends to define an action of $H$ on $G$.  These actions are related by the following proposition.

\begin{proposition}The $\tilde{H}$ action on $\tilde{G}$ is effectively free iff the $H$ action on $G$ is effectively free.  If both are effectively free, the biquotients $\tilde{G}\bq \tilde{H}$ and $G\bq H$ are diffeomorphic.

 \end{proposition}

\begin{proof}  It is easy to see that, in general, the lift of an effectively free action is effectively free, so we focus on the converse.  So, assume the action of $\tilde{H}$ on $\tilde{G}$ is effectively free.  As mentioned previously, the subgroup $SU(2)\subseteq Spin(7)$ projects to $SO(3)\subseteq SO(7)$.  Thus $SU(2)$ must contain $Z(Spin(7))\cong \mathbb{Z}/2\mathbb{Z}$.

Letting $\pi':\tilde{H}\rightarrow H$ denote the double covering, we see that if $\pi'(\tilde{h})\in H$ fixes a point $\pi(\tilde{g})\in G$, then either $\tilde{h}$ fixes $\tilde{g}$, or $\tilde{h}$ maps $\tilde{g}$ to another point in the fiber above $\pi(\tilde{g})$, that is, $\tilde{h} \ast \tilde{g} = z\tilde{g}$.  In the first case, $\tilde{h}$ fixes a point, so acts trivially on $\tilde{G}$, and thus, $h$ acts trivially on $G$.  So, we assume $\tilde{h}\ast \tilde{g} = z\tilde{g}$ for some $z\in Z(Spin(7))\subseteq SU(2)$.  Because $z\in SU(2)\subseteq \tilde{H}$, the point $z^{-1}\tilde{h}\in \tilde{H}$ and fixes $\tilde{g}$.  Thus, $z^{-1}\tilde{H}$ must act trivially on $\tilde{G}$ and so therefore, $\pi'(z^{-1}\tilde{h}) = h$ acts trivially on $G$. 

\

Finally, if both actions are effectively free, $\pi$ induces a covering $\overline{\pi}:\tilde{G}\bq \tilde{H}\rightarrow G\bq H$.  But $g_2:SO(3)\rightarrow SO(7)$ induces an isomorphism of fundamental groups, so $G\bq H$ is simply connected.  It follows that $\overline{\pi}$ is actually a diffeomorphism.

\end{proof}

We may, henceforth focus on biquotients of the form $(SO(7)\times SU(2))\bq (\mathbf{G}_2\times SO(3)\times S^1)$ where the action is induced from a map of the form $f=(f_1,f_2):H\rightarrow G^2$ with $$f_1(A,B,z) = (A, \diag(z^l, \overline{z}^l))\text { and } f_2(A,B,z) = (\diag(A,R(m\theta), R(n\theta)), I)$$ with $\gcd(l,m,n)=1$ and $z = e^{i\theta}$.  By precomposing by the complex conjugation automorphism $S^1\rightarrow S^1$, we assume $l\geq 1$.

\begin{proposition}\label{so7g2free} An action as above is effectively free iff it is free.  Further, it is free iff $l = 1$.

\end{proposition}

\begin{proof}

First, if $l=1$, then the projection of the $S^1$ action to the $SU(2)$ factor of $G$ is free and the projection of the $\mathbf{G}_2\times SO(3)$ action to the $SO(7)$ factor of $G$ is free, so the induced action on $SO(7)\times SU(2)$ is automatically free.

Conversely, assume the action is effectively free.  Recall that, as mentioned in \cite{Ke2}, the maximal torus of $\mathbf{G}_2\subseteq{SO(7)}$ is conjugate to $\diag(R(\phi_1), R(\phi_2), R(\phi_1+\phi_2),1)$.  The maximal torus of $SO(3)\times S^1\subseteq SO(7)$ is conjugate to $\diag(R(\alpha), R(m\theta), R(n\theta), 1)$.

Let $z=e^{i\theta}$ be any $l$-th root of $1$.  Setting $\phi_1 = m\theta$, $\phi_2 = n\theta$, and $\alpha = \phi_1 + \phi_2$, we obtain an element $h\in H$ with $f_1(h) = f_2(h) = (\diag(R(\phi_1), R(\phi_2), R(\phi_1 + \phi_2), 1), I)$.  By Proposition \ref{freetest}, we must have $f_1(h) = f_2(h) \in Z(G)$.  Since $Z(SO(7))$ is trivial, $m\theta = n\theta = 0$ mod $2\pi$.  Thus, every $l$-th root of $1$ must be an $m$-th and $n$-th root of $1$, and so $l|\gcd(m,n)$.  But then $1 = \gcd(l,m,n) = l$.

\end{proof}

Now that we have classified all actions, we classify the quotients.

\begin{proposition}\label{g2sodiff}

Consider the action of $\mathbf{G}_2\times SO(3)\times S^1$ on $SO(7)\times SU(2)$ given by $$(A,B,z)\ast (C,D) = (AC\diag(B,R(m\theta), R(n\theta))^{-1} , \diag(z,\overline{z}) D).$$  The quotient is diffeomorphic to $S^4\times S^2$ when $m$ and $n$ have the same parity, and is diffeomorphic to $S^4\,\hat{\times}\, S^2$ when $m$ and $n$ have opposite parities.

\end{proposition}

\begin{remark}
The map $f:H\rightarrow G^2$ lifts to a map $\tilde{H}\rightarrow \tilde{G}^2$ iff $m$ and $n$ have the same parity.  Thus, according to Proposition \ref{g2sodiff}, when the quotient is diffeomorphic to $S^4\times S^2$, the map lifts, giving half of Theorem \ref{spin7su21}.  On the other hand, when $m$ and $n$ have opposite parities, we can replace $m$, $n$, and $l=1$ by $2m$, $2n$, and $2l = 2$.  This new action has the same orbits as the old one, so the quotient space is diffeomorphic to $S^4\, \hat{\times}\, S^2$.  However, the new action lifts.  This gives the other half of Theorem \ref{spin7su21}.
\end{remark}

\begin{proof}  (Proof of Proposition \ref{g2sodiff})

By the discussion following Proposition \ref{S2S4class}, we know each such biquotient is diffeomorphic to either $S^4\times S^2$ or $S^4\,\hat{\times}\, S^2$.  Using Theorem \ref{swclass}, we now compute the second Stiefel-Whitney class of the tangent bundle to these biquotients, showing that $w_2$ is non-zero precisely when $m$ and $n$ have opposite parities.

We will use a different action (which is equivalent, via Proposition \ref{diffeostabilize}):  $f = (f_1, f_2): H_1\times H_2 = \mathbf{G}_2 \times (SO(3)\times S^1)\rightarrow G^2$ with $$f_1(A,B,z) = (A,I)\text { and } f_2(A,B,C) = (\diag(B, R(m\theta), R(n\theta)), \diag(z, \overline{z}))$$ with $z = e^{i\theta}$.  Then $H$ has full rank in $G$ and the biquotient has the form $H_1\backslash G/H_2$.

Thus, by Theorem \ref{ringmaxrank}, the cohomology ring $H^\ast(G\bq H; \mathbb{Z}_2)$ is isomorphic to the ring $H^\ast(BH_1; \mathbb{Z}_2)\otimes_{H^\ast(BG;\mathbb{Z}_2)} H^\ast(BH_2; \mathbb{Z}_2)$.  We will suppress the coefficient ring $\mathbb{Z}_2$ for the remainder of the proof.

Now, as shown in \cite{Bo1}, $H^\ast(B\mathbf{G}_2)$ has no elements of degree $1$ or $2$.  In addition, $H^\ast(BSU(2))\cong H^\ast(\mathbb{H}P^\infty)$ also has no elements of degree $1$ or $2$.  It follows that we may identify $H^2(G\bq H)$ with $H^2(BH_2)/B{f_2}^\ast(H^2(BSO(7)\times \{I\}))$.  We now work out the map $B{f_2}^\ast$ more explicitly.

We let $Q'$ denote the standard maximal $2$-group of $O(3)\times S^1$, so $Q'\cong \mathbb{Z}_2^4$ is given by $\{\diag(\pm 1,\pm 1,\pm 1)\}\times \{\pm 1\}$.  We let $Q\cong \mathbb{Z}_2^3\subseteq Q'$ denote $Q'\cap H_2$; $Q$ is a maximal $2$-group of $H_2$ \cite{BH1}.  Likewise, we let $R'\cong \mathbb{Z}_2^7$ and $R\cong \mathbb{Z}_2^6 = R'\cap SO(7)$ denote the maximal $2$-group of $O(7)$ and $SO(7)$, respectively.

Consider the basis $\{x_1,x_2,x_3,w\}\subseteq Hom(Q',\mathbb{Z}_2)$ where $x_1$ is dual to the element $(\diag(-1,1,1), I)$ and similarly for $x_2$ and $x_3$, and $w$ is dual to $(I, -1)$.  For any $q\in Q$, the first factor of $q$ has determinant $1$, and so $x_1(q) + x_2(q) + x_3(q) = 0$.  Using this, we may identify $Hom(Q,\mathbb{Z}_2)$ with $Hom(Q', \mathbb{Z}_2)/\sum x_i$.

Using the Leray-Serre spectral sequence associated to the fibration $Q\rightarrow EQ\rightarrow BQ$, we may identify $H^\ast(BQ)$ with $\mathbb{Z}_2[ \overline{x}_i, \overline{w}]/\sum \overline{x}_i$ where $|\overline{x}_i| = |\overline{w}| = 1$, with $dx_i = \overline{x}_i$ and $dw = \overline{w}$.

Likewise, if $y_1,..., y_7\in Hom(R',\mathbb{Z}_2)$ are the duals to the elements of the form $\diag(1,..,-1,...1)\in R'$, then for any $r\in R$, we have $\sum y_i(r) = 0$, so we identify $Hom(R,\mathbb{Z}_2)$ with $Hom(R',\mathbb{Z}_2)/\sum y_i$.  We may also identify $H^\ast(BR)$ with $\mathbb{Z}_2[\overline{y}_i]/\sum \overline{y}_i$ with $|\overline{y}_i| = 1$.

Now, consider the inclusions $Q\rightarrow H_2$ and $R\rightarrow SO(7)$.  These induce maps $H^\ast(BH_2)\rightarrow H^\ast(BQ)$ and $H^\ast(BSO(7))\rightarrow H^\ast (BR)$.  In \cite{Bo1}, Borel shows that these induced maps are injective, identifying $H^\ast(BSO(7))$ and $H^\ast(BH_2)$ with $$\mathbb{Z}_2[\sigma_2(\overline{y}_i), ..., \sigma_7(\overline{y}_i)]/\sum \overline{y}_i \text { and }\mathbb{Z}_2[\sigma_2(\overline{x}_i), \sigma_3(\overline{x}_i), \overline{w}^2]/\sum \overline{x}_i$$ respectively.  Here, $\sigma_j(\overline{y}_i)$ denotes the $j$-th elementary symmetric polynomial in the $\overline{y}_i$ variables.  Thus $H^2(BSO(7))=\mathbb{Z}_2 = \langle \sigma_2(\overline{y}_i)\rangle$ and $H^2(BH_2) = (\mathbb{Z}_2)^2 = \langle \sigma_2(\overline{x}_i), \overline{w}^2\rangle.$

The homomorphism $f_2$ maps $Q$ into $R$, and so, as one can easily verify, $f_2$ induces a map $B{f_2}^\ast:H^\ast(BR)\rightarrow H^\ast(BQ)$, with $B{f_2}(\overline{y}_i) = \begin{cases} \overline{x}_i & i = 1,2,3\\ m\overline{w} & i = 4,5  \\ n\overline{w} & i= 6,7\end{cases}$

Now, an easy calculation shows $B{f_2}^\ast(\sigma_2(\overline{y}_i)) = \sigma_2(\overline{x}_i) + (m+n)^2\overline{w}^2.$  It follows that we may identify $H^2(G\bq H)$ with $\mathbb{Z}_2\langle \sigma_2(\overline{x}_i), \overline{w}^2\rangle/ \sigma_2(\overline{x}_i) + (m+n)^2\overline{w}^2.$  In particular, $0\neq \phi_H^\ast \overline{w}^2\in H^2(G\bq H)$.

The $2$-roots of $SO(7)$ and $SO(3)$ are, according to \cite{BH1}, $\overline{y}_i - \overline{y}_j$ (which is equal to $\overline{y}_i + \overline{y}_2$ mod $2$) and $\overline{x}_i - \overline{x}_j$, each with multiplicity $1$, and $S^1$ has no non-trivial $2$-roots.

Using these $2$-roots, together with Theorem \ref{swclass}, we see $$w(G\bq H) = \phi_H^\ast\left[B{f_2}^\ast \left(\prod_{1\leq i < j \leq 7}(1+\overline{y}_i + \overline{y}_j)\right) \prod_{1\leq k < l \leq 3}(1+\overline{x}_k+\overline{x}_l)^{-1}\right].$$

Because $B{f_2}^\ast(\overline{y}_i) = \overline{x}_i$ for $i = 1,2,3$, this product reduces to $$w(G\bq H) = \phi_H^\ast\left[\prod_{\substack{1\leq i < j \leq 7 \\  3< j} }B{f_2}^\ast (1+\overline{y}_i + \overline{y}_j)\right].$$  Further, since $B{f_2}^\ast(\overline{y}_4)= m\overline{w} = B{f_2}^\ast(\overline{y}_5)$, $$B{f_2}^\ast(1 + \overline{y}_i + \overline{y}_4)(1+\overline{y}_i + \overline{y}_5) = (1+ B{f_2}^\ast\overline{y}_i + m\overline{w})^2 = (1+B{f_2}^\ast \overline{y}_i^2 + m^2 \overline{w}^2)$$ and a similar result holds for $\overline{y}_6$ and $\overline{y}_7$.  Thus, the product reduces to $$\phi_H^\ast\left[\prod_{1\leq i\leq j} (1+B{f_2}^\ast \overline{y}_i^2 + m^2\overline{w}^2)(1+B{f_2}^\ast\overline{y}_i^2 + n^2 \overline{w}^2)\right]$$ and, in particular, \begin{align*} w_2(G\bq H) &= \phi_H^\ast\left[\sum_{1\leq i\leq 7} (B{f_2}^\ast \overline{y}_i^2 +m^2\overline{w}^2) + (B{f_2}^\ast \overline{y}_i^2 + n^2\overline{w}^2)\right]\\ &= 7(m^2+n^2)\phi_H^\ast(\overline{w}^2)\\ &= (m^2+n^2)\phi_H^\ast(\overline{w}^2).\end{align*}

Since we have already shown $\phi_H^\ast \overline{w}^2\neq 0$, we see $w_2(G\bq H)$ is nontrivial iff $m$ and $n$ have different parities.

\end{proof}

\subsection{\texorpdfstring{Biquotients with $\pi_\ast(G\bq H)_\Q\cong \pi_\ast(S^2\times \mathbb{C}P^2)_\Q$}{Rational homotopy groups like S2 x CP^2 }}\label{analysiss2cp2}

In this section, we classify all biquotient actions whose quotients have rational homotopy groups isomorphic to those of $S^2\times \mathbb{C}P^2$, that is, when $$(G,H) \in \{ (SU(3), T^2), (SU(3)\times SU(2), SU(2)\times T^2), (SU(4)\times SU(2), Sp(2)\times T^2)\}.$$
When $(G,H) = (SU(3),T^2)$, the classification has been completed by Eschenburg \cite{Es1}, so we focus on the other two cases in Table \ref{table:gplist}.  To that end, let $G = G_1\times SU(2)$, $H = H_1\times T^2$, with $(G_1,H_1) = (SU(3), SU(2))$ or $(G_1,H_1) = (SU(4), Sp(2)).$

\begin{proposition} If a biquotient action of $H$ on $G$ is effectively free, then the projection of the action of the $H_1$ factor of $H$ to the the $SU(2)$ factor of $G$ is trivial.  Further, the projection of the action to the first factor of $G$ induces the standard action with $G_1/H_1 = S^5$.  In particular, every such biquotient is diffeomorphic to a manifold of the form $S^5\times_{T^2} S^3$ for an effectively free linear $T^2$ action.

\end{proposition}

\begin{proof}

By Proposition \ref{su2biquotients}, if the projection of the action of $H_1$ to the $SU(2)$ factor of $G$ is non-trivial, then $H$ must act freely on $G_1$.  This is impossible for rank reasons, so $H_1$ acts trivially on the $SU(2)$ factor of $G$.

Now, consider the projection of the $T^2$ action on the $SU(2)$ factor of $G$.  Equipping $SU(2)$ with a bi-invariant metric, which has positive curvature, the $T^2$ action is isometric.  It follows from \cite{GSe} that there is a point in $SU(2)$ with isotropy group containing an $S^1\subseteq T^2$.

It follows that the projection of the $H_1\times S^1$ action to $G_1$ must be effectively free, giving a $4$-dimensional biquotient.  These have been classified \cite{KZ,DeV1,Es2} - the only examples with $H$ a product of a semi-simple group with a circle have $G_1\bq H_1$ equal to either the homogeneous space $SU(3)/SU(2) = S^5$, or equal to $SU(4)/Sp(2)$.  But, recalling the canonical double cover $SU(4)\rightarrow SO(6)$ restricts to a double cover $Sp(2)\rightarrow SO(5)$, we identify $SU(4)/Sp(2) = SO(6)/SO(5) = S^5$.  Thus, in either case, $G_1/H_1 = S^5$.

\end{proof}

We now concretely identify the $T^2$ action on $S^5\times S^3$.  In the hardest case, when $(G_1,H_1) = (SU(4), Sp(2))$, we identify $SU(4)/Sp(2)$ with $SO(6)/SO(5) = S^5$ using the double cover $SU(4)\rightarrow SO(6)$.  Recall the construction of this double cover: the standard representation $\mathbb{C}^4$ of $SU(4)$ induces a representation of $SU(4)$ on $\Lambda^2 \mathbb{C}^4$.  This $6$-dimensional representation has real type, defining a map $SU(4)\rightarrow SO(6)$, the double cover.  In particular, under this double covering map, the maximal torus of $SU(4)$, consisting of matrices of the form $\diag(z_1, z_2, z_3, \overline{z}_1 \overline{z}_2 \overline{z}_3)$, maps, up to conjugacy, to $\diag(R(\theta_2 + \theta_3), R(\theta_1 + \theta_3), R(\theta_1 + \theta_2))$ where $z_j = e^{i\theta_j}$.

\begin{proposition}\label{convert}(1)  Under the usual identification $SU(2)\cong S^3$, the $T^2$ biquotient action on $SU(2)$ given by $$(z,w)\ast B = \diag(z^{k_1} w^{l_1}, \overline{z}^{k_1}\overline{w}^{l_1}) B \diag(z^{k_2} w^{l_2}, \overline{z}^{k_2}\overline{w}^{l_2})^{-1}$$ is equivalent to the $T^2$ action on $S^3 = \{b=(b_1,b_2)\in \mathbb{C}^2:|b| = 1\}$ given by $$(z,w)\ast b = (z^{k_1 - k_2} w^{l_1-l_2} b_1, z^{k_1 + k_2} w^{l_1+l_2} b_2).$$

(2)  Under the usual identification of $SU(3)/SU(2)$ with $S^5 = \{a = (a_1, a_2, a_3)\in \mathbb{C}^3:|a| = 1\}$, the $T^2$ biquotient action on $S^5$ induced from $SU(2)\times T^2\rightarrow SU(3)^2$ with $$(A,z,w)\mapsto \left(\diag(\overline{z}^{m_1} \overline{w}^{n_1} A, z^{2m_1} w^{2n_1}) , \diag(z^{m_2}w^{n_2}, z^{m_3} w^{n_3}, \overline{z}^{m_2 + m_3}\overline{w}^{n_2 + n_3})\right)$$ is given by $$(z,w)\ast a = (z^{2m_1 - m_2} w^{2n_1-n_2}a_1,  z^{2m_1-m_3} w^{2n_1-n_3}a_2,  z^{2m_1 + m_2 + m_3} w^{2n_1 + n_2 + n_3}a_3).$$

(3)  The $T^2$ biquotient action on $S^5 = SU(4)/Sp(2)$ induced from $Sp(2)\times T^2\rightarrow SU(4)^2$ with $$(A,z,w) \mapsto \left( A, \diag(z^{m_1} w^{n_1} , z^{m_2} w^{n_2}, z^{m_3}w^{n_4}, \overline{z}^{m_1+m_2+m_3}\overline{w}^{n_1 + n_2 + n_3})\right)$$ has the form $$(z,w)\ast a = (z^{m_2 + m_3} w^{n_2 + n_3} a_1, z^{m_1+m_3}w^{n_1 + n_3} a_2, z^{m_1+m_2} w^{n_1 + n_2} a_3).$$

\end{proposition}

It follows easily from this proposition that every linear $T^2$ action on $S^5\times S^3$ is orbit equivalent to one induced from a biquotient of the form $(SU(3)\times SU(2))\bq (SU(2)\times T^2)$ and to one induced from a biquotient of the form $(SU(4)\times SU(2))\bq (Sp(2)\times T^2)$.

For the remainder of this section, we focus on linear $T^2$ actions on $S^5\times S^3$.  By reparamaterizing $T^2$, we may assume that any action has the form $$(z,w)\ast(a,b) = (z^{m_1} a_1, z^{m_2}w^{n_2}a_2, z^{m_3}w^{n_3} a_3, w^{l_1} b_1, z^{k_2} w^{l_2} b_2).$$  Further, we may assume $\gcd(m_1, m_2, m_3, k_2) = \gcd(n_2, n_3, l_1, l_2) = 1$.  By using the diffeomorphisms of $S^5\times S^3$ mapping $a_i\rightarrow \overline{a_i}$ or $b_j$ to $\overline{b_j}$, as well as the automorphism $S^1\rightarrow S^1$ with $z\mapsto \overline{z}$, we may assume all $m_i$, $l_j$, and $k_2$ are non-negative.

\begin{proposition}\label{cp2s2actions} Such an action is effectively free iff $m_1 = m_2 = m_3 = l_1 = l_2 = 1$ and up to interchanging $n_2$ and $n_3$, $$(n_2, n_3, k_2)\in \{(0,0, k_2), (n_2, n_3,0), (1,1,2), (2,2,1), (0,2,1), (0,1,2)\}.$$  Further, such an action is effectively free iff it is free.

\end{proposition}

\begin{proof}

The key observation is that if $(z,w)\in T^2$ fixes a point $(a,b)\in S^5\times S^3$, then it fixes a point with every coordinate $1$ or $0$.  So, it is enough to focus on these six points.

Let $z$ be an $m_1$-th root of 1.  Then $(z,1)$ fixes $(1,0,0,1,0)\in S^5\times S^3$, and thus, must fix every point of $S^5\times S^3$.  This implies $m_1$ divides each of $m_2,m_3$, and $k_2$.  In particular, $1 = \gcd(m_1, m_2, m_3, k_2) = |m_1|$.  An analogous argument shows that $m_2 = m_3 = l_1 = l_2 = 1$.

On easily sees that these conditions on $m_i$ and $l_j$ are necessary and sufficient to guarantee the $T^2$ action moves points with either $a_1 = 1$ or $b_1 = 1$.  It now follows that any nontrivial $(z,w)\in T^2$ moves the point $(1,0,0,1,0)$, so the action cannot have ineffective kernel.  In particular, it is free iff effectively free.

Finally, we seek conditions on $n_2$, $n_3$, and $k_2$ which are necessary and sufficient to guarantee every non-identity element of $T^2$ action moves $(0,1,0,0,1)$ and $(0,0,1,0,1)$.

Now, $(z,w)\ast(0,1,0,0,1) = (0,1,0,0,1)$ iff $zw^{n_1} = 1 = z^{k_2} w$.  There is only one solution to this equation, $(z,w) = (1,1)$, iff $|1-k_2 n_2|=1$, so $|1-k_2 n_2| = 1$ is necessary to have a free action.  The same argument using the point $(0,0,1,0,1)$ shows $|1-k_2 n_3| = 1$ is necessary to have a free action.

It is easy to see that the only integer solutions to $|1-k_2 n_3| = |1-k_2 n_2| = 1$ are, up to order,  $$(n_2, n_3, k_2)\in \{(0,0, k_2), (n_2, n_3,0), (1,1,2), (2,2,1), (0,2,1), (0,1,2)\}.$$ 

\end{proof}

We will use the notation $A(n_2, n_3, k_2)$ to denote these biquotients.  Consider, first, biquotients of the form $A(n_1, n_2, 0)$.  Since $k_2 = 0$, the circle factor of $T^2$ given by the $z$-coordinate, $S^1_z$, acts only on $S^5$.  Further, the other factor, $S^1_w$, acts freely on the $S^3$ with quotient $S^2$ and the action of $S^1_w$ on $S^5$ commutes with the $S^1_z$ action.  In terms of the concrete description $G_1\times SU(2) = SU(3)\times SU(2)$ with $H = SU(2)\times T^2 = SU(2)\times S^1_z \times S^1_w$, $SU(2)\times S^1_z$ acts freely on $G_1$ with quotient $\mathbb{C}P^2$, $S^1_w$ acts freely on $SU(2)$, and the $S^1_w$ action on $G_1$ normalizes the $SU(2)\times S^1_z$ action.

In other words, the biquotients $A(n_2,n_3,0)$ are all decomposable, naturally having the structure of $\mathbb{C}P^2$-bundles over $S^2$.  Each such bundle is associated to the Hopf bundle $S^1\rightarrow S^3\rightarrow S^2$ via the $S^1_w$ action on $\mathbb{C}P^2$, so the structure group is a circle.

If we equip $G = G_1\times SU(2)$ with a bi-invariant metric, the resulting metric on $G_1\bq (SU(2)\times S^1_z) = \mathbb{C}P^2$ is the Fubini-Study metric, and the $S^1_w$ action is isometric.  In particular, $S^1_w\subseteq K = Iso(\mathbb{C}P^2) = SU(3)/(\mathbb{Z}/3\mathbb{Z})$.

With $S^1$ acting on $S^3$ via the Hopf map and on $K$ via the inclusion, we can form the space $S^3\times_{S^1} K$; projection onto the first factor gives this space the structure of a principal $K$ bundle over $S^2$.  Note then that $A(n_2,n_3,0)\cong \mathbb{C}P^2\times_{S^1} S^3 \cong \mathbb{C}P^2 \times_{K}(S^3\times_{S^1} K)$ (see, for example, \cite{Mi1}).

Now, principal $K$ bundles over $S^2$ are in one to one correspondence with $[S^2, BK]$, homotopy classes of maps from $S^2$ into $BK$.  Since $\pi_1(K) = \mathbb{Z}/3\mathbb{Z}$, $\pi_2(BK) = \mathbb{Z}/3\mathbb{Z}$.  In particular, $S^3\times_{S^1} K$ is one of three principal bundles.  Hence, biquotients of the form $A(n_1,n_2, 0) = \mathbb{C}P^2 \times_{K}(S^3\times_{S^1} K)$ fall into at most three diffeomorphism types.  In fact, we will show the two non-trivial bundles have diffeomorphic total spaces, so biquotients of the form $A(n_1,n_2,0)$ fall into only two diffeomorphism types.

In a similar manner, one can see the biquotients $A(0,0,k_2)$ are decomposable, having the structure of an $S^2$-bundle over $\mathbb{C}P^2$ with structure group $S^1$ acting linearly.  Further, since every principal $S^1$-bundle over $\mathbb{C}P^2$ is one of $S^1\times \mathbb{C}P^2$, $S^5$, or a lens space, all of which are homogeneous, we see that the total space of every $S^2$ bundle over $\mathbb{C}P^2$ with structure group $S^1$ is of the form $P\times_{S^1} S^2$ with $P$ a biquotient.  But $P\times_{S^1} S^2$ is obviously a biquotient.  Thus, we have shown the following proposition.

\begin{proposition}\label{s2buncp2}

The total space of every $S^2$-bundle over $\mathbb{C}P^2$ with the structure group reducing to $S^1$ acting linearly on $S^2$ is a biquotient.

\end{proposition}

We now state the diffeomorphism classification.

\begin{theorem}\label{s2cp2diffeo} With the following exceptions, the homotopy types of the two biquotients $A(n_2,n_3, k_2)$ and $A(n_2', n_3', k_2')$ are distinct unless, up to permuting $n_2$ and $n_3$, we have $(|n_2|, |n_3|, |k_2|) = (|n_2'|, |n_3'|, |k_2'|)$.

(1)  $A(n_2, n_3, 0)$ is diffeomorphic to $ A(n_2', n_3', 0)$ if $n_2 + n_3 = \pm( n_2' + n_3')$ mod $3$.

(2)  $A(0,2,1)$, $A(2,2,1)$, and $A(0,0,1)$ are diffeomorphic.

(3)  $A(0, 1, 2)$, $A(1,1,2)$, and $A(0,0,2)$ are diffeomorphic.

\end{theorem}

In order to prove Proposition \ref{s2cp2diffeo}, we first compute cohomology rings and characteristic classes.  The results are summarized in the following proposition.

\begin{proposition}\label{s2cp2top}  With $N = n_2 + n_3$, the cohomology ring $H^\ast(A(n_2,n_3,k_2))$ is isomorphic to $\mathbb{Z}[\overline{u}_1, \overline{u}_2]/I$ where $|\overline{u}_i|  =2$ and $I$ is an ideal generated by two relations given by columns $2$ and $3$ in Table \ref{table:s2cp2top}.  Further, the first Pontryagin class, as well as second and fourth Stiefel Whitney classes are listed, where, for $w_2$ and $w_4$, $\overline{u}_1$ and $\overline{u}_2$ are understood to be taken mod $2$.

\begin{table}[h]

\begin{center}

\begin{tabular}{|c|c|c|c|c|c|}

\hline

$(n_2,n_3,k_2)$ & $\overline{u}_2^2 =$ & $\overline{u}_1^3 = $ & $p_1$ & $w_2$ & $w_4$ \\

\hline

\hline

$(0,0,k_2)$ & $ -k_2 \overline{u}_1 \overline{u}_2$ & $ 0$ & $(k_2^2 + 3) \overline{u}_1^2$ & $(k_2 + 1) \overline{u}_1$ & $(k_2+1) \overline{u}_1^2 $\\

\hline

$(n_2,n_3, 0)$ & $ 0$ & $-N \overline{u}_1^2 \overline{u}_2$ & $3\overline{u}_1^2 + 2N \overline{u}_1 \overline{u}_2 $ & $ \overline{u}_1 + N\overline{u}_2 $ & $ \overline{u}_1^2 $ \\

\hline

$(1,1,2)$ & $-2\overline{u}_1 \overline{u}_2 $ & $0$ & $7 \overline{u}_1^2$ & $\overline{u}_1 $ & $ \overline{u}_1^2 $ \\

\hline

$(2,2,1)$ & $-\overline{u}_1 \overline{u}_2$ & $0$ & $4\overline{u}_1^2 $ & $ 0$ & $ 0 $ \\

\hline

$(0,1,2)$ &  $-2\overline{u}_1 \overline{u}_2$ & $ -\overline{u}_1^2 \overline{u}_2 $ & $7\overline{u}_1^2$ & $ \overline{u}_1 + \overline{u}_2 $ & $ \overline{u}_1^2 $ \\

\hline

$(0,2,1)$ & $ -\overline{u}_1 \overline{u}_2 $ & $-2\overline{u}_1^2 \overline{u}_2$ & $4\overline{u}_1^2$ & $ 0$ & $ 0 $\\

\hline

\end{tabular}\caption{Topology of $A(n_2,n_3, k_2)$}\label{table:s2cp2top}

\end{center}

\end{table}

\end{proposition}

It follows easily from this proposition that the cohomology ring of $A(n_2,n_3,k_2)$ is torsion free.  We will eventually see the ring structure completely characterize these examples up to diffeomorphism.  Our key tool for proving this is the following theorem \cite{Ju,Wal,Zu}.

\begin{theorem}(Jupp, Wall, Zubr)\label{6dimclass}

Suppose $M_1$ and $M_2$ are closed simply connected $6$-manifolds whose cohomology rings are torsion free.  Then $M_1$ is diffeomorphic to $M_2$ iff there is an isomorphism between their integral cohomology rings which maps the first Pontryagin of $M_1$ to that of $M_2$ and whose mod $2$ restriction maps the Stiefel-Whitney classes of $M_1$ to those of $M_2$.

\end{theorem}

Using the methods of Section \ref{linearsphere}, we now compute the cohomology rings and characteristic classes of $A(n_2,n_3,k_2)$, proving Proposition \ref{s2cp2top}.  In fact, the two matrices $A_1$ and $A_2$ describing the maps $H_1(T^2)\rightarrow H_1(T_{U(3)\times U(2)})$ are $A_1 = \begin{bmatrix} 1 & 0 \\ 1 & n_2 \\ 1 & n_3\end{bmatrix}$ and $A_2 = \begin{bmatrix} 0 & 1 \\ k_2 & 1\end{bmatrix}.$  Hence, Proposition \ref{typcalc} gives the cohomology ring of $A(n_2,n_3,k_2)$ as $$\mathbb{Z}[\overline{u}_1, \overline{u}_2]/\langle \sigma_3(A_1^t u), \sigma_2(A_2^t u) \rangle = \mathbb{Z}[\overline{u}_1, \overline{u}_2]/\langle \overline{u}_1(\overline{u}_1 + n_2 \overline{u}_2)(\overline{u}_1 + n_3 \overline{u}_2), \overline{u}_2(k_2 \overline{u}_1 + \overline{u}_2)\rangle.$$  Substituting in the admissible value of $(n_2,n_3,k_2)$ and simplifying gives the first three columns of Table \ref{table:s2cp2top}.

We now compute the characteristic classes of $A(n_2,n_3,k_2)$ in the following proposition.  Substituting the admissible values for $(n_2,n_3,k_2)$ and using appropriate relations fills in the rest of Table \ref{table:s2cp2top}.

\begin{proposition}\label{s2cp2calc}  The characteristic classes $p_1$, $w_2$, and $w_4$ of $A(n_2,n_3,k_2)$ are as follows: \begin{align*} p_1 &= (k_2^2 + 3) \overline{u}_1^2 + (n_2^2 + n_3^2)\overline{u}_2^2 + 2(n_1 + n_2) \overline{u}_1 \overline{u}_2 \\ w_2 &= (k_2 + 1) \overline{u}_1 + (n_2 + n_3)\overline{u}_2 \\ w_4 &= (k_2 + 1)\overline{u}_1^2 + n_2 n_3\overline{u}_2^2 + k_2(n_2+n_3)\overline{u}_1 \overline{u}_2.\end{align*}  For $w_2$ and $w_4$, $\overline{u}_1$ and $\overline{u}_2$ are taken mod $2$.

\end{proposition}

\begin{proof}We use the notation of Section \ref{linearsphere}: for $G = U(3)\times U(2)$ and $H = T^2\times U(2)\times U(1)$, we let  $\overline{x}_1, \overline{x}_2, \overline{x}_3, \overline{y}_1, \overline{y}_2 \in H^2(BT_G)$ be the transgressions of the elements of $H^1(T_G)$ which are dual to the canonical generators of $H_1(T_G)$, and similarly for $\overline{u}_1,\overline{u}_2, \overline{s}_1, \overline{s}_2, \overline{t}\in H^2(BT_H)$.

The map $Bf_1^\ast:H^2(BT_G)\rightarrow  H^2(BT_H)$ is computed just as in Section \ref{linearsphere}.  In particular, we have $Bf_1^\ast \overline{x}_i = \overline{x}_i A_1^t $ and likewise, $Bf_1^\ast \overline{y}_i =\overline{y}_i A_2^t$.

Further, as is seen in Section \ref{linearsphere}, from the Leray-Serre spectral sequence associated to the fibration $G\rightarrow G\bq H\rightarrow BH$, we note that $\phi_H^\ast$ identifies $\sigma_1(\overline{s})$ with $ \sigma_1( \overline{u} A_1^t ) = 3\overline{u}_1 + (n_2 + n_3)\overline{u}_2$.  Likewise, $\sigma_2(\overline{s})$ is identified with $\sigma_2(\overline{u} A_2^t ) = \overline{u}_1(2\overline{u}_1 + (n_2 + n_3)\overline{u}_2) + (\overline{u}_1 + n_2 \overline{u}_2)(\overline{u}_1 + n_3 \overline{u}_2)$  and $\overline{t}$ is identified with $\sigma_1(\overline{u} A_2^t ) = k_2 \overline{u}_1 + 2 \overline{u}_2.$

It follows that $(\overline{s}_1 - \overline{s}_2)^2 = (\overline{s}_1 + \overline{s}_2)^2 - 4 \overline{s}_1\overline{s}_2$ is identified with $(3\overline{u}_1 + (n_2 + n_3)\overline{u}_2)^2 -4(\overline{u}_1(2\overline{u}_1 + (n_2 + n_3)\overline{u}_2) + (\overline{u}_1 + n_2 \overline{u}_2)(\overline{u}_1 + n_3 \overline{u}_2))$ which simplifies to $-3\overline{u}_1^2 -2(n_2 + n_3) \overline{u}_1\overline{u}_2 + (n_2-n_3)^2 \overline{u}_2^2.$

Then, by equation \eqref{firstpont}, we have \begin{align*} p_1 =& \phi_H^\ast \left[ Bf_1^\ast \left(\sum_{1\leq i < j\leq 3}(\overline{x}_i-\overline{x}_j)^2\right) +  Bf_1^\ast (\overline{y}_1-\overline{y}_2)^2 - (\overline{s}_1 - \overline{s}_2)^2\right] \\ =& n_2^2 \overline{u}_2^2 + n_3^2\overline{u}_2^2 + (n_2 - n_3)^2\overline{u}_2^2 + k_2^2 \overline{u}_1^2 \\ &-\left(-3\overline{u}_1^2 -2(n_2 + n_3) \overline{u}_1\overline{u}_2 + (n_2-n_3)^2 \overline{u}_2^2\right)\\ =& (k_2^2 + 3)\overline{u}_1^2 + (n_2^2 + n_3^2)\overline{u}_2^2 + 2(n_2+n_3)\overline{u}_1 \overline{u}_2. \end{align*}

\

We now compute the Stiefel-Whitney classes.  Note that $w_6 = 0$, being mod $2$ reduction of the Euler characteristic.  Using Proposition \ref{easysw}, we see \begin{align*} w &= \phi_H^\ast\left(Bf_2^\ast\left[\prod_{1\leq i < j \leq 3}(1+\overline{x}_i +\overline{x}_j)\right]Bf_2^\ast(1 + \overline{y}_1 + \overline{y}_2) \left( 1 + \overline{s}_1 + \overline{s}_2\right)^{-1}\right)\\ &=  \phi_H^\ast \left((1 + \overline{s}_1 + \overline{s}_2)(1+\overline{s}_1)(1+\overline{s}_2)(1+\overline{t})(1+\overline{s}_1 + \overline{s}_2)^{-1}\right)\\ &= \phi_H^\ast\left( (1 + \overline{s}_1)(1+\overline{s}_2)(1+\overline{t}) \right)\\  &= \phi_H^\ast\left(1 + \overline{s}_1 + \overline{s}_2 +  \overline{t} + \overline{s}_1 \overline{s}_2 + (\overline{s}_1 + \overline{s}_2) \overline{t} \right) \\  &=   1 + (k_2+1)\overline{u}_1 + (n_2 + n_3)\overline{u}_2 + (k_2 + 1)\overline{u}_1^2 + n_2 n_3 \overline{u}_2^2 + k_2(n_2+n_3) \overline{u}_1\overline{u}_2.  \end{align*}

\end{proof}

We now analyze the cohomology rings in more detail.

\begin{proposition}\label{s2cp2ring}  
With the following exceptions, two biquotients $A(n_2,n_3, k_2)$ and $A(n_2', n_3', k_2')$ have non-isomorphic cohomology rings unless, up to permuting $n_2$ and $n_3$, $(|n_2|, |n_3|, |k_2|) = (|n_2'|, |n_3'|, |k_2'|)$.

(1)  $H^\ast(A(n_2, n_3, 0))\cong H^\ast(A(n_2', n_3', 0))$ if $n_2 + n_3 = \pm( n_2' + n_3')$ mod $3$.

(2)  $H^\ast(A(0,2,1))\cong H^\ast(A(0,0,1))\cong H^\ast(A(2,2,1))$.

(3)  $H^\ast(A(0, 1, 2))\cong H^\ast(A(0,0,2)) \cong H^\ast(A(1,1,2))$.

Further, for every exceptions, there is an isomorphism which carries characteristic classes to characteristic classes.

\end{proposition}

Upon proving Proposition \ref{s2cp2ring}, we will have, with the aid of Theorem \ref{6dimclass}, proved Theorem \ref{s2cp2diffeo}.

\begin{proof}(Proof of Proposition \ref{s2cp2ring})

We first prove the exceptions.  If $(n_2', n_3', 0)$ is another triple with $n_2 + n_3 = \epsilon( n_2' + n_3') \pmod{3}$ with $\epsilon\in\{\pm1\}$, set $\lambda = \frac{2}{3}((n_2' + n_3')\epsilon - n_2 - n_3)\in \mathbb{Z}$.  Then the map $H^2(A(n_2,n_3,0))\rightarrow H^2(A(n_2',n_3', 0))$ given by $\overline{u}_1 \mapsto \epsilon \overline{u}_1 + \lambda \overline{u}_2$ and $\overline{u}_2\mapsto \overline{u}_2$ is easily seen to extend to a characteristic class preserving isomorphism, proving exception (1).

The isomorphism from $H^\ast(A(2,2,1))$ to $H^\ast(A(0,0,1))$ is obvious as is the one from $H^\ast(A(1,1,2))$ to $H^\ast(A(0,0,2))$.  The isomorphism from $H^\ast(A(0,0,1))$ to $H^\ast(A(0,2,1))$ with $\overline{u}_1 \mapsto \overline{u}_1 + 2\overline{u}_2$ and $\overline{u}_2\mapsto -\overline{u}_2$ is easily seen to map characteristic classes to characteristic classes.  Finally, the map $H^2(A(0,0,2))\rightarrow H^2(A(0,1,2))$ with $\overline{u}_1\mapsto \overline{u}_1 + \overline{u}_2$ and $\overline{u}_2\mapsto -\overline{u}_2$ extends to a characteristic class persevering isomorphism.  This handles exceptions (2) and (3).

We now show these are the only exceptions.  Since we have already shown each of the sporadic examples is diffeomorphic to either $A(0,0,1)$ or $A(0,0,2)$, we may focus on the families. That is, we show first that $H^\ast(A(n_1, n_2,0))$ is not isomorphic to $H^\ast(A(0,0,k_2))$ except when $k_2 = 0$, second, that $H^\ast(A(n_1,n_2,0))\not\cong H^\ast(A(n_1',n_2', 0))$ when $n_1 + n_2 \neq \pm (n_1' +n_2') \pmod{3}$, and finally, that $H^\ast(A(0,0,k_2)) \not\cong H^\ast(A(0,0,k_2'))$  if $|k_2| \neq |k_2'|$.  

In the rings $H^\ast(A(n_1,n_2,0))$, there is a nontrivial element which squares to $0$, $\overline{u}_2^2 = 0$.  However, $H^\ast(A(0,0,k_2))$ with $k_2\neq 0$ does not share this property.  For, if $(a\overline{u}_1 + b\overline{u}_2)^2 = 0$, then, because there are no relations on $\overline{u}_1^2$ in Table \ref{table:s2cp2top}, $a = 0$.  Then $(b\overline{u}_2)^2 = -b^2 k_2\overline{u}_1 \overline{u}_2  = 0$ iff $b = 0$ or $k_2 = 0$.

We now focus on the $A(n_1,n_2,0)$ family.  Due to exception (1), if $n_1+n_2 = 0 \pmod{3}$, then $A(n_1,n_2,0)\cong A(0,0,0)$, and if $n_1 + n_2\neq 0 \pmod{3}$, then $A(n_1,n_2,0)\cong A(1,0,0)$.  So, we need only show $H^\ast(A(1,0,0))$ is not isomorphic to $H^\ast(A(0,0,0))$.  First notice that in both rings, the only degree two elements which square to $0$ are multiples of $\overline{u}_2$.  For $A(1,0,0)$, the equation $0 = (a\overline{u}_1 + b\overline{u}_2)^3 = a^2(3b-a)\overline{u}_1^2 \overline{u}_2$ is solved only when $a=0$ or $a = 3b$.  Thus, multiples of $3\overline{u}_1 + \overline{u}_2$ are characterized by cubing to $0$ but not squaring to $0$.  For $A(0,0,0)$, degree two elements whose square is nonzero but whose cube is $0$ are all multiples of $\overline{u}_1$.  It follows that any isomorphism $H^\ast(A(0,0,0))\rightarrow H^\ast(A(1,0,0))$ must map $\overline{u}_2$ to $\pm \overline{u}_2$ and $\overline{u}_1$ to $\pm(3\overline{u}_1 + \overline{u}_2)$.  But then $\overline{u}_1$ is clearly not in the image of this map.

\

To finish the proof, we need only distinguish the rings of the form $H^\ast(A(0,0,k_2))$.  We first note that the only nontrivial elements of $H^2(A(0,0,k_2))$ which cube to $0$ are multiples of $\overline{u}_1$.  To see this, we compute using the relations $\overline{u}_1^3 = 0$ and $\overline{u}_2^2 = -k_2 \overline{u}_1 \overline{u}_2$ that $(a\overline{u}_1 + b\overline{u}_2)^3 = b(3a^2 - 3abk_2 + b^2 k_2^2) \overline{u}_1^2\overline{u}_2$.  If $b\neq 0$, then this vanishes iff $3a^2 - 3abk_2 + b^2 k_2^2 = 0$.  But the discriminant of this polynomial is $-3b^2k_2^2 < 0$, so there is no solution with $a$ real and $b\neq 0$.

It follows that any  isomorphism $\phi:H^\ast(A(0,0,k_2))\rightarrow H^\ast(A(0,0,k_2'))$ must map $\overline{u}_1$ to $\pm \overline{u}_1$.  On $H^2$, $\phi$, which can be described by an element of $Gl(2,\mathbb{Z})$, is thus given by $\phi(\overline{u}_1) = \epsilon \overline{u}_1$, $\phi(\overline{u}_2) = \lambda \overline{u}_1 + \delta \overline{u}_2$ with $\epsilon,\delta \in \{\pm 1\}$.  In order for $\phi$ to be well defined, we must have $\phi(\overline{u}_2)^2 = k_2' \phi(\overline{u}_1) \phi(\overline{u}_2)$.  Expanding this and the equating the $\overline{u}_1^2$ and $\overline{u}_1\overline{u}_2$ components, one gets the pair of equations $$ \lambda^2 = \epsilon k_2' \lambda \text{ \ and \ } \epsilon\delta k_2' = 2\lambda\delta - k_2.$$ The first equation implies $\lambda = 0$ or $\lambda = \epsilon k_2'$.  Substituting either option into the second equation, one finds $|k_2| = |k_2'|$.

\end{proof}

\subsection{\texorpdfstring{Biquotients with $\pi_\ast(G\bq H)_\Q\cong \pi_\ast((S^2)^3)_\Q$}{Rational homotopy groups like S2 x S2 x S2 }}
\label{analysiss23}

In this subsection, $G$ will denote $SU(2)^3$ and $H$ will denote $T^3$.  Equipping $SU(2)$ with its bi-invariant metric, the $H$ action is by isometries and hence linear.  It turns out that, just as in the previous section, classifying effectively free actions is easier done by directly studying linear actions on spheres.  As in the previous section, we interpret $S^1 = \{z\in\mathbb{C}: |z| = 1\}$ and $S^3 = \{(a_1,a_2)\in \mathbb{C}^2: |a_1|^2 + |a_2|^2 = 1\}$.

We note that the biquotient action of $T^3$ on $SU(2)$ given by $ (z_1, z_2, z_3)\ast A =$ $$\begin{bmatrix}z_1^{m_1} z_2^{m_2} z_3^{m_3} & \\ &  \overline{z}_1^{m_1} \overline{z}_2^{m_2}\overline{z}_3^{m_3}\end{bmatrix} \cdot A \cdot \begin{bmatrix}z_1^{n_1} z_2^{n_2} z_3^{n_3} & \\ &  \overline{z}_1^{n_1} \overline{z}_2^{n_2} \overline{z}_3^{n_3}\end{bmatrix}^{-1}$$ is equivalent to the action of $T^3$ on $S^3$ given by $$(z_1, z_2, z_3)\ast(a_1,a_2) = (z_1^{m_1-n_1}z_2^{m_2-n_2}z_3^{m_3-n_3} a_1, z_1^{m_1+n_1}z_2^{m_2 + n_2} z_3^{m_3+n_3} a_2)$$ under the isomorphism $SU(2)\rightarrow S^3$ mapping $\begin{bmatrix} a_1 & a_2\\ -\overline{a}_2 & \overline{a}_1\end{bmatrix}$ to $(a_1,a_2)$.

We now focus on linear $T^3$ actions on $(S^3)^3$.

\begin{proposition}\label{t3act}  Suppose $T^3$ acts on $(S^3)^3$ freely.  Then, up to the modifications in Proposition \ref{diffeostabilize}, the action has the form $$(z_1,z_2,z_3)\ast \big((a_1,a_2), (b_1,b_2), (c_1,c_2)\big) =$$ $$ \big( (z_1 a_1, z_1 z_2^{k_2} z_3^{k_3} a_2), (z_2 b_1, z_1^{l_1} z_2 z_3^{l_3} b_2), (z_3 c_1, z_1^{m_1} z_2^{m_2} z_3 c_2)\big).$$  Further, if these exponents are put into a matrix as $A=\begin{bmatrix} 1 & k_2 & k_3\\ l_1 & 1 & l_3\\ m_1 & m_2 & 1\end{bmatrix}$, then $\det(A) = \pm 1$ and each of the three $2\times 2$ diagonal cofactors also has determinant $\pm 1$.  Conversely, every such $3\times 3$ matrix defines a free action of $T^3$ on $(S^3)^3$.

\end{proposition}

\begin{proof}
Totaro \cite{To2} has proven that a biquotient action in the form of the proposition is free iff each of the listed determinant conditions holds, so we need only show that we can modify any action to have that form.

Up to equivalence, a general linear $H$ action on $(S^3)^3$ takes the form $(z_1,z_2, z_3)\ast\big( (a_1,a_2), (b_1, b_2), (c_1, c_2)\big) = $ $$\big( (z_1^{\alpha_1} z_2^{\beta_1} z_3^{\gamma_1} a_1, z_1^{k_1} z_2^{k_2} z_3^{k_3} a_2), (z_1^{\alpha_2} z_2^{\beta_2} z_3^{\gamma_2} b_1, z_1^{l_1} z_2^{l_2} z_3^{l_3} b_2), (z_1^{\alpha_3} z_2^{\beta_3} z_3^{\gamma_3} c_1 , z_1^{m_1} z_2^{m_2} z_3^{m_3} c_2)\big).$$ 

Let $X =\begin{bmatrix} \alpha_1 & \alpha_2 & \alpha_3\\ \beta_1 & \beta_2 & \beta_3\\ \gamma_1 & \gamma_2 & \gamma_3\end{bmatrix}.$  Since the action is effective, $\det(X)\neq 0$.  If $\det(X) \neq \pm 1$, then there is a non-integral, rational vector $r = (r_1, r_2, r_3)^t$ for which $Xr$ is integral.  Then $(e^{2\pi i r_1 }, e^{2\pi i r_2 }, e^{2\pi i r_3 })\in H$ fixes $(1,0)^3 \in (S^3)^3$, contradicting freeness of the action, so $\det(X) = \pm 1$.

It follows that $X^{-1}$ is integral.  By precomposing $f$ with the isomorphism $X^-1:T^3\rightarrow T^3$, which is nothing but a reparametrization, the action now has the desired form on the $a_1, b_1,$ and $c_1$ coordinates.  That is, we may assume $X$ is the identity matrix.  Thus, we only need to show that the power of $z_i$ on the coordinates $a_2$, $b_2$ and $c_2$ is $1$.

But, if $z_1$ is a $k_1$-th root of $1$, then the point $(z_1,1,1)\in T^3$ fixes the point $\big( (0,1) , (1,0), (1,0)\big)\in (S^3)^3$.  Since the action is free, this implies $z_1 = 1$.  That is, every $k_1$-th root of $1$ is $1$.  Thus, $k_1 = \pm 1$ and, by simultaneously precomposing by the automorphism of $T^3$ $z_1\mapsto \overline{z}_1$ and composing with the diffeomorphism of $S^3$ given by $(a_1,a_2)\rightarrow (\overline{a}_1, a_2)$, we may assume $k_1 = 1$.  Likewise, we see $l_2 = m_3 = 1$.

\end{proof}

We now need to classify all $3\times 3$ matrices of the form $A = \begin{bmatrix} 1 & k_2 & k_3\\ l_1 & 1 & l_3\\m_1 & m_2 & 1\end{bmatrix}$ with $\det(A) = \pm 1$ and all diagonal cofactors equal to $\pm 1$.

\begin{proposition}\label{matlist}  Suppose $T^3$ acts freely on $(S^3)^3$ with characteristic matrix $A = \begin{bmatrix} 1& k_2 & k_3 \\ l_1 & 1 & l_3\\ m_1 & m_2 & 1\end{bmatrix}$ which has determinant $\pm 1$ and with diagonal cofactors each having determinant $\pm 1$.  Then, up to equivalence, $A$ belongs to one of the three infinite families $$R(m_1,m_2) = \begin{bmatrix} 1 & 2 & 0 \\ 1 & 1 & 0 \\ m_1 & m_2 & 1\end{bmatrix},\, S(k_3,l_3) = \begin{bmatrix} 1 & 2 & k_3 \\ 1 & 1 & l_3 \\ 0 & 0 & 1 \end{bmatrix},\text{ and } T(l_1,m_1,m_2)= \begin{bmatrix} 1 & 0 & 0 \\ l_1 & 1 & 0 \\ m_1 & m_2 & 1\end{bmatrix}$$ or $A$ is one of four other sporadic examples: $$A_1 = \begin{bmatrix}1 & 2 & 2 \\ 1 & 1 & 2\\ 1 & 1 & 1\end{bmatrix}, A_2 = \begin{bmatrix}1 & 2 & 0 \\ 1 & 1 & 2\\ 1 & 1 & 1 \end{bmatrix}, A_3 = \begin{bmatrix}1 & 2 & 0 \\ 1 & 1 & 1 \\ 2 & 2 & 1\end{bmatrix}, \text{ and } A_4 = \begin{bmatrix}1 & 2 & 2 \\ 1 & 1 & 2 \\ 1 & 0 & 1\end{bmatrix}.$$

\end{proposition}

Biquotients in the $R$, $S$, and $T$ family are are all decomposable:  for the $R$ family we see the $T^2_{z_1,z_2} \subseteq T^3 = H$ which given by the $z_1$ and $z_2$ coordinates acts freely on the first two factors of $G$ with $SU(2)^2\bq T^2\cong \mathbb{C}P^2 \# \mathbb{C}P^2$ \cite{To1,DeV1}.  The $S^1_{z_3}\subseteq T^3$ given by the $z_3$ coordinate acts only on the last factor of $G$ freely, with quotient $S^2$.  Further, the $T^2_{z_1,z_2}$ action on the last factor of $G$ centralizes the $S^1_{z_3}$ action.  So, each biquotient in the $R$ family is decomposable, naturally begin $S^2$ bundles over $\mathbb{C}P^2\#\mathbb{C}P^2$.  In a similar fashion, for the $S$ family, the $S^1_{z_3}$ action normalizes the $T^2_{z_1,z_2}$ action, so each biquotient in the $S$ family is decomposable, being the total space of a $\mathbb{C}P^2 \# \mathbb{C}P^2$ bundle over $S^2$.

Every biquotient in the $T$ family is decomposable in two ways.  First, it is the total space of an $S^2$ bundle over $S^2\times S^2$ if $l_1$ is even and an $S^2$ bundle over $\mathbb{C}P^2\#-\mathbb{C}P^2$ if $l_1$ is odd.  To see this, note the $T^2_{z_1, z_2}$ action on the first two factors is free with quotient either $S^2\times S^2$ or $\mathbb{C}P^2\#-\mathbb{C}P^2$, depending on the parity of $l_1$ \cite{DeV1}.  The $S^1_{z_3}$ action on the last factor of $G$ is free with quotient $S^2$, and the $T^2_{z_1, z_2}$ action on the last factor of $G$ centralizes this action.  Second, each biquotient in the $T$ family is the total space of a bundle over $S^2$ with fiber either $S^2\times S^2$ if $m_2$ is even and with fiber $\mathbb{C}P^2\#-\mathbb{C}P^2$ if $m_2$ is odd.  To see this, notice the $T^2_{z_2,z_3}$ action on the second two factors of $G$ is free with quotient $S^2\times S^2$ or $\mathbb{C}P^2\#-\mathbb{C}P^2$, depending on the parity of $m_2$ \cite{DeV1}.  The $S^1_{z_1}$ action on the first factor of $G$ is free with quotient $S^2$ and the $S^1_{z_1}$ action on the second two factors of $G$ centralizes the $T^2_{z_2,z_3}$ action.

Finally, we point out that the four sporadic examples are not naturally decomposable.  The $T$ family of biquotients have been previously studied by Totaro \cite{To2}, who showed that among these examples one can find infinitely many non-isomorphic rational cohomology rings.

\begin{proof}(Proof of Proposition \ref{matlist})

The outer automorphisms given by simultaneously swapping factors of $G$ and $H$ corresponds, in the matrix description, to simultaneously swapping a pair of rows and the same pair of columns, that is, to conjugating by a transposition matrix.  The inner automorphism given by simultaneously mapping $(a_1,a_2)\mapsto (\overline{a}_1, \overline{a}_2)$ and $z_1\mapsto \overline{z}_1$ correspond to conjugation of $A$ by the diagonal matrix with a $-1$ in the first entry, and $1$s along the rest of the diagonal; one has a similar result for the other factors of $G$ and  $H$.

Since the determinant of the top left $2\times 2$ block is $\pm 1$, using these equivalences,  we see that either $k_2 = 0$ or $(l_1,k_2) = (1,2)$.

When $(l_1,k_2) = (1,2)$ with all other entries at most $2$ in absolute value, one has a relatively small list of options.  It is easy to see that, after taking equivalences into account, and after removing those examples in the $S$ family, that is, with $m_1 = m_2= 0$ one obtains $A_1, A_2, A_3,$ and $A_4$ as well as seven extra examples:

$$\begin{bmatrix}1 & 2 & 1 \\ 1 & 1 & 0\\ 0 & 2 & 1\end{bmatrix}, \begin{bmatrix}1 & 2 & 1 \\ 1 & 1 & 0\\ 2 & 2 & 1\end{bmatrix}, \begin{bmatrix} 1 & 2 & 2 \\ 1 & 1 & 0\\ 0 & 1 & 1 \end{bmatrix}, \begin{bmatrix} 1 & 2 & 2 \\ 1 & 1 & 0\\ 1 & 1 & 1\end{bmatrix}, \begin{bmatrix} 1 & 2 & 0 \\ 1 & 1 & 1 \\ 1 & 0 & 1\end{bmatrix}, \begin{bmatrix} 1 & 2 & 0 \\ 1 & 1 & 1 \\ 1 & 2 & 1\end{bmatrix}, \begin{bmatrix} 1 & 2 & 2\\ 1 & 1 & 2\\ 0 & 1 & 1\end{bmatrix}.$$

If one interchanges $a_1$ and $a_2$, and then reparamaterizes the action to have the form of Proposition \ref{t3act}, then this has the effect of changing the matrix $$\begin{bmatrix} 1 & k_2 & k_3\\l_1 & 1 & l_3\\ m_1 & m_2 & 1\end{bmatrix} \text{ to } \begin{bmatrix}1 & -k_2 (1-k_2 l_1) & -k_3(1-k_3 m_1)\\ l_1 & 1 & (l_3-k_3 l_1)(1-k_3 m_1)\\ m_1 & (m_2-k_2 m_1)(1-k_2 l_1) & 1\end{bmatrix}.(\ast)$$  Applying $(\ast)$ to the first extra matrix $\begin{bmatrix} 1 & 2 & 1 \\ 1 & 1 & 0\\ 0 & 2  & 1\end{bmatrix}$ and then negating the third row and column, we obtain the new matrix $\begin{bmatrix} 1 & 2 & 1\\ 1 & 1 & 1 \\ 0 & 2 & 1\end{bmatrix}$, which is equivalent to the 6th extra matrix, so these two matrices describe equivalent actions.  Further, if we now interchange the first two rows and columns of the new matrix and the apply $(\ast)$, we see that this action is equivalent to the action determined by $A_1$.

Similarly, one can show the remaining 5 extra matrices all describe actions equivalent to those of $A_i$.  This concludes the case where $k_2 l_1 = 2$ with all entries at most $2$ in absolute value.

\

Now, assume $k_2 = 0$ and $l_1$ is arbitrary.  If $m_2 = 0$ or if $k_3 = 0$, then, after interchanging rows and columns, it is easy to see that this just gives matrices in the $R$, $S$, or $T$ family.  So, we assume $m_2$ and $k_3$ are both non-zero, and, without loss of generality, that $k_3 > 0$.

Since $k_3 > 0$, the equation $|1-k_3 m_1| = 1$ forces either $m_1 = 0$ or $(k_3,m_1) = (1,2)$ or $(2,1)$.  If $m_1 = 0$, then up to equivalence, one obtains the matrix $\begin{bmatrix} 1 & 0 & 1\\ 2 & 1 & 0 \\ 0 & -1 &1\end{bmatrix}$.  Upon applying $(\ast)$, multiplying the third row and column by $-1$, interchanging first and second row and colulmn, and then interchanging the second and third row and column, one gets $\begin{bmatrix}1 & 2 & 2 \\ 1 & 1 & 0\\ 0 & 1 & 1\end{bmatrix}$, one of the extra examples above.

Hence, we assume $m_1 \neq 0$, so $(k_3,m_1)$ is, up to order, $(1,2)$.  If $l_3 = 0$, then since $\det(A) = \pm 1$, we must have $l_1 = 0$, giving the $S$ family, or $(k_3, l_1, m_2) =(2,1,1)$ up to order, giving the matrix $\begin{bmatrix} 1 & 0 & 2 \\ 1 & 1 & 0 \\ 1 & 1 & 1\end{bmatrix}$.  Interchanging appropriate rows and columns, this matrix has $k_2 l_1 = 2$ with all other entries at most $2$ in abolute value, so is equivalent to an $A_i$ or $S$ example.

Finally, assume $l_3\neq 0$.  The equation $\det(A) = \pm 1$ then implies that $(k_3,l_1,m_2)$ is, up to order, either $(2,1,1)$, $(2,2,1)$, or $(4,1,1)$, where $l_1 = 4$ in the last case.  In the first two cases, one can interchange appropriate rows and columns to obtain a matrix with $k_2l_1 = 2$ and all entries bounded by $2$ in absolute value, so these matrices are equivalent to one defining an $A_i$ or $S$ biquotient.  When $l_1 = 4$, we obtain the matrix $\begin{bmatrix} 1 & 0 & 1 \\ 4 & 1 & 2\\ 2 & 1 & 1\end{bmatrix}$.  Interchanging the first and third rows and columns, and then applying $(\ast)$, we obtain $\begin{bmatrix} 1 & 1 & 2\\ 2 & 1 & 0 \\ 1 & 1 & 1\end{bmatrix}$, which has $k_2 l_1 = 2$ with all entries boudned by $2$ in absolute value, so is equivalent to an $A_i$ or $S$ example.

\end{proof}

We now prove the analogue of Proposition \ref{s2buncp2}, using an approach suggested by an anonymous referee.

\begin{proposition}\label{bundoverb4}  Suppose $E$ is the total space of a bundle $S^2\rightarrow E\rightarrow B$ with $B = S^2\times S^2$ or $B = \mathbb{C}P^2\# \pm \mathbb{C}P^2$ where the structure group reduces to $S^1$ acting linearly on $S^2$.  Then $E$ is diffeomorphic to a biquotient.

\end{proposition}

\begin{proof}
According to \cite{To1}, each possibility for $B$ is diffeomorphic to a biquotient of the form $(SU(2))^2\bq T^2$.  If $P\rightarrow B$ is a principal $S^1$ bundle with Euler class an indivisible element of $H^2(B)\cong\mathbb{Z}^2$, then it is clear that $P$ is of the form $P = (SU(2))^2\bq S^1$ for an appropriate $S^1\subseteq T^2$.  If the Euler class is divisible, then $P$ has the form $(SU(2)^2) \times_{S^1} S^1$ where the circle action on $S^1$ is given by $z\ast w = z^k w$, and is, in particular, a biquotient.

Now, if $E\rightarrow B$ is an $S^2$-bundle with structure group reducing to $S^1$ acting linearly, then $E$ is of the form $P\times_{S^1} S^2$.  This action is clearly a biquotient action.
\end{proof}

\

We now state the diffeomorphism and homotopy equivalence classification of biquotients of the form $(SU(2))^3\bq T^3$.

\begin{proposition}\label{s3modt3}

Other than the five exceptions listed below, two biquotients of the form $SU(2)^3\bq T^3$ which are defined by distinct effectively free actions are not homotopy equivalent.

(1)  $A_2$ is diffeomorphic to $R(0,1)$ and $A_4$ is diffeomorphic to $R(0,2)$.

(2)  $R(m_1,m_2)$ is diffeomorphic to $R(m_1', m_2')$ if $m_1^2 + (m_2-m_1)^2 = m_1'^2 + (m_2'-m_1')^2$ and $m_i = m_i' \pmod{2}$.

(3)  $S(k_3,l_3)$ is diffeomorphic to $S(k_3', l_3')$ if $$(k_3', l_3') \in \{\pm(k_3 ,l_3 ),\pm(k_3 - 2l_3 ,-l_3 ),\pm(-k_3 ,l_3 - k_3 ),\pm (k_3 - 2l_3 ,k_3 - l_3 )\}.$$

(4)  If $m_2(2m_1 - m_2 l_1) = m_2' (2m_1' - m_2'l_1')=0$, then $T(l_1, m_1, m_2)$ is diffeomorphic to $T(l_1', m_1',m_2')$ if $l_1 = l_1' \pmod{2}$ and $m_2 = m_2' \pmod{2}$ or if $l_1= m_2 + 1\pmod{2}$ and $l_1' = m_2' + 1\pmod{2}.$

(5)  $T(l_1, m_1,m_2)$ is diffeomorphic to $T(l_1',m_1',m_2')$ if $l_1 = l_1'\pmod{2}$ and $m_2(m_1 - \lfloor \frac{l_1}{2}\rfloor m_2) = m_2'(m_1' - \lfloor \frac{l_2}{2} \rfloor m_2')$.

\end{proposition}

In (5), the notation $\lfloor \cdot \rfloor$ denotes the greatest integer function.

We will only prove this theorem for a subset these biquotients, the rest of the arguments being similar.  More specifically, we will classify the diffeomorphism types of the four sporadic example as well as the $R$ family.  We will also show that no biquotient belonging to a family is homotopy equivalent to a biquotient in another family, except in the obvious case $R(0,0) \cong S(0,0)\cong S^2\times (\mathbb{C}P^2\#\mathbb{C}P^2)$.

The method of proof closely mirrors that of Section \ref{analysiss2cp2}.  We first compute the cohomology rings and characteristic classes of these examples; we will see that the cohomology groups are isomorphic to those of $(S^2)^3$, so are torsion free.  In particular, we may then apply Theorem \ref{6dimclass} of Jupp, Wall, and Zubr to see that two such biquotients are diffeomorphic iff there is an isomorphism of their cohomology rings which maps characteristic classes to characteristic classes.

\begin{proposition}\label{s3t3data}  The cohomology ring of the biquotient $G\bq H$ associated to the integers $(k_2,k_3, l_1,l_3, m_1,m_2)$ is given by $\Z[u_1,u_2,u_3]/I$ where $|u_i|=2$ and $I$ is the ideal generated by $u_1(u_1+k_2 u_2 + k_3 u_3 )$, $u_2(l_1 u_1 + u_2 + l_3 u_3)$, and $u_3(m_1 u_1 + m_2 u_2 + u_3)$.  The first Pontryagin class $p_1(G\bq H)$ is $(k_2 u_2 + k_3 u_3)^2 + (l_1 u_1 + l_3 u_3)^2 + (m_1 u_1 + m_2 u_2)^2$.  Further, $H^\ast(G\bq H; \Z_2) = \Z_2[u_1,u_2,u_3]/I$ and the Stiefel-Whitney classes are given by $w_2 = (l_1 + m_1)u_1 +(k_2 + m_2)u_2 + (k_3+l_3)u_3$ and $w_4 = (k_2 u_2 + k_3 u_3)(l_1 u_1 + l_3 u_3) +(k_2 u_2 + k_3 u_3)(m_1 u_1 + m_2 u_2) + (l_1 u_1 + l_3 u_3)(m_1 u_1 + m_2 u_2 )$.

\end{proposition}

The proof of this proposition follows the method outlined in Section \ref{linearsphere}, which was done in detail in in the proof of Proposition \ref{s2cp2calc}.  Thus, we omit it.

From this proposition, it follows that $H^2(G\bq H) \cong \mathbb{Z}^3$ and all odd degree cohomology groups vanish.  Poincar\'e duality then implies $H^4(G\bq H)$ is isomorphic to $\mathbb{Z}^3$ as well, so $H^\ast(G\bq H) \cong H^\ast((S^2)^3)$ as groups.  In particular, the cohomology rings of these examples are all torsion free.

We now prove a portion of Proposition \ref{s3modt3}.

\begin{proof}(Proof of a part Proposition \ref{s3modt3})

For $A_2$, the map sending $u_1\in H^\ast(R(0,1))$ to $u_1+u_2+2u_3$, $u_2$ to $u_1+2u_3$, and $u_3$ to $u_3$ is easily seen to extend to a ring isomorphism carrying characteristic classes to characteristic classes.  Likewise, for $A_4$, one can use the map sending $u_1\in H^\ast(R(0,2))$ to $u_1+u_2+u_3$, $u_2$ to $u_1 + u_3$, and $u_3$ to $u_3$.  By Theorem \ref{6dimclass}, it follows that $A_2$ and $R(0,1)$ are diffeomorphic, as are $A_4$ and $R(0,2)$.  This proves exception 1 of Proposition \ref{s3modt3}.

\

We now distinguish the three families.  We claim that in the $R$ family, no nontrivial degree $2$ element squares to $0$, while in the $S$ family, there is a non-trivial family of elements which square to $0$, and in the $T$ family, there are at least two families of elements which square to $0$.

Begin with a degree $2$ element $x = \alpha u_1 + \beta u_2 + \gamma u_3 \in H^2(G\bq H)$.  Computing, we find $x^2$ is given by $$\begin{cases} u_1u_2(2\alpha \beta -2\alpha^2 - \beta^2) + u_1u_3(2\alpha\gamma - m_1\gamma^2) + u_2u_3(2\beta\gamma -m_2\gamma^2) & R \text{ biquotients}\\ u_1u_2(2\alpha \beta-2\alpha^2 - \beta^2) + u_1 u_3(2\alpha\gamma -k_3\alpha^2) + u_2 u_3(2\beta\gamma -l_3\beta^2) & S \text{ biquotients} \\ u_1u_2(2\alpha \beta - l_1\beta^2) + u_1u_3(2\alpha\gamma -m_1\gamma^2) + u_2 u_3(2\beta\gamma -m_2\gamma^2) & T \text{ biquotients} \end{cases}.$$

But $2\alpha \beta - 2\alpha^2 -\beta^2 = -\alpha^2 - (\alpha-\beta)^2$ so vanishes iff $\alpha = \beta = 0$.  From this, one easily sees that $x^2 = 0$ for a nontrivial element in the $R$ family only when $m_1 = m_2 = 0$, which is the same as the entry in the $S$ family when $k_3 = l_3 = 0$.  Further, for the $S$ family, $x^2 = 0$ iff $\alpha = \beta = 0$, so there is a single family of solutions.

Finally, for the $T$ biquotients, one can easily verify that, other than when $\gamma = \beta = 0$, one has the additional family of solutions given by $\gamma = 0$ and $l_1\beta = 2\alpha$.  If $m_2 \neq 0$ and $2m_1\neq l_1 m_2$, it is easy to see that all elements which square to $0$ fall into one of these two families.  On the other hand, if $m_2 = 0$, one has the additional family with $\alpha = \beta = 0$ while if  $2m_1 = l_1 m_2$, one has the additional family with $l_1\beta = 2\alpha = m_1\gamma$.

In summary, a biquotient belonging to a family (other than $R(0,0)$) is not even homotopy equivalent to a biquotient in another family because their cohomology rings are non-isomorphic.  Further, those biquotients in the $T$ family with $m_2(2m_1 - l_1 m_2) = 0$, that is, with $p_1 = 0$ are distinct up to homotopy from the rest of the biquotients in the $T$ family.

\

We now show $A_1$ and $A_3$ are distinct, up to homotopy, from every other biquotient of the form $SU(2)^3\bq T^3$.  For both examples, a simple calculation shows that no non-trivial degree $2$ element squares to $0$, so if $A_1$ or $A_3$ is homotopy equivalent to a biquotient in a family, it must be the $R$ family.  However, in the $R$ family, there are two elements of a basis of $H^2$, namely, $u_1$ and $u_2$, with the property that $u_1^2 = -2 u_1 u_2 = 2 u_2^2$.

We claim that for neither $A_1$ nor $A_3$ is there a pair of elements $u_1', u_2'$ in a basis of $H^2$ with $u_1'^2 = -2u_1' u_2' = 2 u_2'^2$.  For example, in the cohomology ring of $A_1$, if one sets $u_1' = \alpha u_1+\beta u_2+\gamma u_3$ and $u_2' = \delta u_1 + \epsilon u_2 + \zeta u_3$, then the $u_1 u_3$ component of the equation $u_1'^2 = 2 u_2'^2$, taken mod $2$ implies that $\beta$ is even.  Likewise, the $u_1 u_3$ component implies $\gamma$ is even, which then implies, via the $u_2 u_3$ component, that $\zeta$ is even.  Finally, looking at the $u_1 u_3$ component mod $4$ then forces $\alpha$ to be even, so $u_1$ cannot be a member of $\mathbb{Z}$-basis of $H^2(A_1)$.  A similar argument works for $A_3$.

We must also show show that $A_1$ and $A_3$ are distinct up to homotopy.  But, for $A_1$, $p_1$ is a generator while for $A_3$, $p_1$ is $10$ times a generator.  Hence, taken mod $24$ they do not coincide.  Since $p_1 \pmod{24}$ is a homotopy invariant \cite{AH}, $A_1$ and $A_3$ are not homotopy equivalent.

\

We now distinguish the rings in the $R$ family individually, proving exception 2 of Proposition \ref{s3modt3}.  Consider two $R$ biquotients $R(m_1,m_2)$ and $R(m_1', m_2')$.

 To begin with, we make the substitution $u_1' = u_1+u_2$.  In the $(u_1',u_2,u_3)$ coordinates, the cohomology rings of the $R$ biquotients have the form $\Z[u_1',u_2,u_3]/I$ where $I$ is generated by $(u_1')^2 = u_2^2$, $u_1'u_2 = 0$, and $u_3^2 = -m_1 u_1'u_2 - (m_2-m_1) u_2 u_3$.  A simple calculation shows that $u_1'$ and $u_2$ are, up to reordering and changing signs, the only two primitive elements in $H^2$ satisfying $(u_1')^2 = u_2^2$ and $u_1' u_2 = 0$, and further, that the map swapping these two elements extends to a characteristic class preserving isomorphism of $R(m_1,m_2)$.  Thus, we may assume that any isomorphism $f:R(m_1, m_2)\rightarrow R(m_1', m_2')$ maps $u_1'\in R(m_1, m_2)$ to $u_1'\in R(m_1', m_2')$ and likewise for $u_2$.  This, in turn, implies that $f(u_3) = u_3 + \alpha u_1' + \beta u_2\in R(m_1', m_2')$ for some integers $\alpha$ and $\beta$.

In order for this to be well defined, we must have $$f(u_3)^2 = -m_1 f(u_3) u_1' - (m_2-m_1) f(u_3) u_2.$$  Inspecting the $(u_1')^2$, $u_1' u_3$, and $u_2 u_3$ components of this equation gives the three equations \begin{align*} \alpha^2 + \beta^2 &= -cm_1 \alpha - (m_2-m_1)\beta \\ 2\alpha + -m_1' &= -m_1 \\ 2\beta - (m_2'-m_1') &= -(m_2 - m_1).\end{align*}  The second and third equations clearly imply $m_i = m_i' \pmod{2}$.  Solving the second equations for $\alpha$ and $\beta$ and substituting into the first and rearranging then gives $m_1^2 + (m_2 - m_1)^2 = m_1'^2 + (m_2'-m_1')^2$ as claimed.  Hence these conditions are necessary.  A simple calculation then shows that these conditions are also sufficient to define an isomorphism and that this isomorphism carries characteristic classes into characteristic classes.

\end{proof}

\subsection{\texorpdfstring{Biquotients with $\pi_\ast(G\bq H)_\Q\cong \pi_\ast(S^3\times S^4)_\Q$}{Rational homotopy groups like S3 x S4  }}
\label{analysiss3s4}

In this section, we completely classify all biquotients with rational homotopy groups isomorphic to those of $S^3\times S^4$.

Let $(G,H) = (G_1\times SU(2), H_1\times SU(2))$ denote one of the four entries in Table \ref{table:gplist} having the same rational homotopy groups as $S^3\times S^4$.  That is, $$(G_1,H_1)\in  \{(SU(4), SU(3)), (Sp(2), SU(2)), (Spin(7), \mathbf{G}_2), (Spin(8), Spin(7))\}.$$

  We have already classified the form of the action in Proposition \ref{S2S4class}:  in every case but one, the first factor of $H$ acts on one side of the first factor of $G$ with quotient $S^7$.  The $SU(2)$ factor of $H$ then acts via the Hopf action on $S^7$ and either trivially or by conjugation on the $SU(2)$ factor of $G$.  Thus, in the non-exceptional case, the biquotients fall into at most 2 diffeomorphism types, each having the form $S^7\times_{SU(2)} S^3$ where $SU(2)$ action on $S^7$ is the Hopf action and the $SU(2)$ action on $S^3$ is either trivial or by conjugation.  We will later see these two biquotients are not even homotopy equivalent.
	
The exceptional case occurs when $G = Sp(2)\times SU(2)$.  In this case, there is an additional action where the first factor of $H$ acts on the first factor of $G$ with quotient $T^1 S^4$.  The other $SU(2)$ factor of $H$ then acts freely on $T^1 S^4$ with quotient $S^4$ and it acts either trivially or by conjugation on the $SU(2)$ factor of $G$.  When it acts trivially, the quotient is again $S^3\times S^4$.  However, when it acts by conjugation, we get a biquotient which is not homotopy equivalent to either of the two biquotients of the form $S^7\times_{SU(2)} S^3$.  Hence, there are precisely three homotopy types and three diffeomorphism types of biquotients which have the rational homotopy groups of $S^3\times S^4$.
	
	In each case except $G = Spin(8)\times SU(2)$, the $H_1$ action on $G_1$ is unique.  But when $G_1 = Spin(8)$, there are, up to conjugacy, three non-trivial homomorphisms $Spin(7)\rightarrow Spin(8)$: the standard inclusion, as well as the two spin representations.  However, the triality automorphism of $Spin(8)$ interchanges the images of these three homomorphims, and so, up to the equivalence in Proposition \ref{diffeostabilize}, there is a unique $H_1$ action on $G_1$ in this case as well.
	
	Summarizing, we have the following proposition.
	
	\begin{proposition}
	
	Suppose $G\bq H$ is a reduced biquotient whose rational homotopy groups are isomorphic to those of $S^3\times S^4$.  Then, up to the equivalence in Proposition \ref{diffeostabilize}, the $H$ action on $G$ is given as follows with $\epsilon \in \{0,1\}$.
	
	For $(G,H) = (SU(4)\times SU(2), SU(3)\times SU(2))$, we have $$(A,B)\ast(C,D) = (\diag(A,1) C \diag(B,B)^{-1}, B^\epsilon D B^{-\epsilon}).$$
	
	For $(G,H) = Sp(2)\times SU(2), SU(2)^2)$, we have, after identifying $SU(2) = Sp(1)$, $$(p,q)\ast (A, r) = (\diag(p,p) A \diag(q,1)^{-1}, p^\epsilon rp^{-\epsilon})$$ or the exceptional action $$(p,q)\ast(A,r) = (\diag(p,p) A \diag(q,1)^{-1}, q^{\epsilon} r q^{-\epsilon}).$$
	
	For $(G,H) = (Spin(7)\times SU(2), \mathbf{G}_2\times SU(2))$, we have $$(A,B)\ast(C,D) = (A C \diag(B,1,1,1,1), B^\epsilon D B^{-\epsilon})$$ where $\diag(B,1,1,1,1)^{-1}$ indicates the lift of the standard inclusion $SO(3)\rightarrow SO(7)$ to $Spin(7)$.
	
	For $(G,H) = (Spin(8)\times SU(2), Spin(7)\times SU(2))$, we have $$(A,B)\ast(C,D) = ( \diag(A,1) C \diag(B,B)^{-1}, B^{\epsilon} D B^{-\epsilon} )$$ where $\diag(B,B)$ indicates the lift of the map $SU(2)\rightarrow \Delta SO(4)\subseteq SO(4)\times SO(4)\subseteq SO(8)$.
	
	\end{proposition}
	
As mentioned previously, there are at most three diffeomorphism types arising.  We now show that each of these three biquotients are not even homotopy equivalent to each other.  Let $X$ denote $S^7\times_{SU(2)} S^3$ where the $SU(2) = S^3$ acts on $S^7$ via the Hopf map and on $SU(2)$ by conjugation and let $Y$ denote $T_1 S^4\times_{SU(2)} S^3$.

\begin{proposition}\label{calc1}  The cohomology rings $H^\ast(X)$, $H^\ast(Y)$, and $H^\ast(S^3\times S^4)$ are isomorphic, but the first Pontryagin class is $p_1(X) = \pm 4 \in H^4(X)\cong \Z$ and $p_1(Y) = \pm 8 \in H^4(Y)\cong \Z$.  In particular, since $p_1$ mod $24$ is a homotopy invariant \cite{AH}, $X$, $Y$, and $S^3\times S^4$ are distinct up to homotopy.

\end{proposition}

\begin{proof}

We will only do the computation for $X$, the computation for $Y$ being analogous.

To carry out the calculation, we will describe $X$ and $Y$ as biquotients with $G = Sp(2)\times Sp(1)$ and $H = Sp(1)\times Sp(1)$.  We use the maximal tori $$T_H = \{(w_1, w_2):w_i \in S^1\} \text{ and } T_G = \left\{ \left(\begin{bmatrix} u_1 & \\ & u_2\end{bmatrix}, v\right): u_i, v\in S^1\right\}.$$  We then have $H^\ast(BT_H) \cong \Z[\overline{w}_1,\, \overline{w}_2]$ where $\deg(\overline{w}_i) = 2$ and $H^\ast(BH)\cong  \Z[\overline{w}_1^2,\, \overline{w}_2^2].$  Likewise, since $T_{G\times G} = T_G\times T_G$, we have $H^\ast(BT_{G\times G})\cong H^\ast(BT_G)\otimes H^\ast (BT_G)$ where $H^\ast(BT_G)$ is isomorphic to $\Z[\overline{u}_1,\, \overline{u}_2,\, \overline{v}]$ and hence, that $H^\ast(BG) \cong \Z[\overline{u}_1^2 + \overline{u}_2^2,\, \overline{u}_1^2 \overline{u}_2^2,\, \overline{v}^2].$  We also let $H^\ast(G) = \Lambda_\Z(s_3, s_6)\otimes \Lambda_\Z(t_3)$ where $\deg(s_i) = i$ and $\deg(t_3) = 3$.  Then by Proposition \ref{differentials}, in the spectral sequence associated to the fibration $G\rightarrow BG\rightarrow BG\times BG$, we have \begin{align*} ds_3 &= \left(\overline{x}_1^2 + \overline{x}_2^2\right)\otimes 1 - 1\otimes\left(\overline{x}_1^2 + \overline{x}_2^2\right)\\ ds_7 &= \overline{x}_1^2 \, \overline{x}_2^2 \otimes 1 - 1\otimes \overline{x}_1^2\, \overline{x}_2^2  \\ dt_3 &= \overline{y}^2 \otimes 1 - 1 \otimes \overline{y}^2.\end{align*}

The map $f:H\rightarrow G\times G$ defining the biquotient action is given by $f(p,q) = (f_1(p,q), f_2(p,q))$ where $$f_1(p,q) = \left(\begin{bmatrix} p& \\ & q\end{bmatrix}, p\right) \text{ and }f_2(p,q) = \left(\begin{bmatrix} 1 & \\ & 1\end{bmatrix}, p\right).$$  Hence, we see the maps $Bf_i^\ast:H^2(BT_G)\rightarrow H^2(BT_H)$ are given by \begin{center}

\begin{tabular}{ccc}

$Bf_1^\ast(\overline{u}_1) = \overline{w}_1$ & & $Bf_2^\ast(\overline{u}_1) = 0$\\

$Bf_1^\ast(\overline{u}_2) = \overline{w}_2$ & & $Bf_2^\ast(\overline{u}_2) = 0$\\

$Bf_1^\ast(\overline{v}) = \overline{w}_1$ & & $Bf_2^\ast(\overline{v}) = \overline{w}_1.$\end{tabular}

\end{center}

Thus, in the spectral sequence for the fibration $G\rightarrow G\bq H\rightarrow BH$, we have \begin{center} \begin{tabular}{rrr} $ds_3$ = & $Bf^\ast\left((\overline{u}_1^2 + \overline{u}_2^2)\otimes 1 - 1\otimes(\overline{u}_1^2 + \overline{u}_2^2)\right)$ = & $\overline{w}_1^2 + \overline{w}_2^2$ \\

$ds_7=$ & $Bf^\ast\left(\overline{u}_1^2 \, \overline{u}_2^2 \otimes 1 - 1\otimes \overline{u}_1^2\, \overline{u}_2^2\right)$ =& $\overline{w}_1^2\, \overline{w}_2^2$\\

$dt_3=$ & $Bf^\ast\left(\overline{v}^2 \otimes 1 - 1 \otimes \overline{v}^2\right)$ =& $0 $.

\end{tabular}

\end{center}

Computing the spectral sequence for the fibration of the biquotient, we see that $E^\infty_{0,3} = \Z$ generated by $t_3$, but all other $E^\infty_{p,q}$ with $0< p+q \leq 3$ are trivial.  Hence, $H^3(X)\cong \Z$.  It follows from Poincar\'e duality that $H^\ast(X)\cong H^\ast(S^3\times S^4)$ as rings.  Further, we see that $H^4(X)\cong E^\infty_{4,0}$ is given by $\Z_{\langle\overline{w}_1^2\rangle}\oplus \Z_{\langle \overline{w}_2^2 \rangle}/(\overline{w}_1^2 + \overline{w}_2^2)\cong \mathbb{Z}.$  From this, we see $\phi_H^\ast (\overline{w}^2_2)$ is a generator of $H^4(X)$.

\

We now compute the first Pontryagin class of $X$.  The positive roots of $G$, $\Delta^+ G$, are given by $u_1 + u_2$, $u_1 - u_2$, $2u_1$, $2u_2$, and $2v$ while the positive roots of $H$ are $2w_1$ and $2w_2$.  Recall that, as mentioned after Theorem \ref{pclass}, that $\phi_G^\ast =  \phi_H^\ast Bf_2^\ast.$  Using this together with equation \eqref{firstpont} and the fact that $Bf_2^\ast(\overline{u}_i) = 0$, we compute \begin{align*} p_1(X) &= \phi_H^\ast( Bf_2^\ast(4v^2) - 4\overline{w}_1^2 - 4\overline{w}_2^2)\\ &= -4\phi_H^\ast(\overline{w}_2^2)\\ &= \pm 4\in H^4(X)\cong \mathbb{Z}.\end{align*}

\end{proof}

\subsection{\texorpdfstring{Biquotients with $\pi_\ast(G\bq H)_\Q\cong \pi_\ast(S^3\times \mathbb{C}P^2)_\Q$}{Rational homotopy groups like S3 x CP2  }}\label{analysiss3cp2}
In this section, we present the classification of reduced biquotients with rational homotopy groups isomorphic to $S^3\times \mathbb{C}P^2$.  We will completely classify the effectively free actions, but we are unable to give a full diffeomorphism classification.

According to Table \ref{table:gplist}, the rational homotopy groups of a reduced biquotient can be isomorphic to those of $S^3\times \mathbb{C}P^2$ only when $(G,H) = (SU(3),S^1), (SU(3)\times SU(2), SU(2)\times S^1), (SU(4)\times SU(2), Sp(2)\times S^1)$.  When $(G,H) = (SU(3),S^1)$, the biquotients were discovered by Eschenburg \cite{Es1} who classifies the actions.  The diffeomorphism, homeomorphism, and homotopy equivalence classification of Eschenburg spaces is incomplete, but many partial results are known \cite{Kru1,Kru2,Kru3,KS1,KS2}.

When $(G,H) = (SU(4)\times SU(2), Sp(2)\times S^1)$, there is a unique non-trivial homomorphism $Sp(2)\rightarrow SU(4)$ with quotient $SU(4)/Sp(2) = S^5$.  Further, there is no non-trivial homomorphism $Sp(2)\rightarrow SU(2)$.  It follows that such biquotients have the form $(SU(4)/Sp(2))\times_{S^1} SU(3) = S^5\times_{S^1} S^3$ where the $S^1$ action on $S^5$ and $S^3$ is linear.

When $(G,H) = (SU(3)\times SU(2), SU(2)\times S^1)$, the situation is more complicated because there are two non-trivial homomorphisms from $SU(2)$ into $SU(3)$, the natural inclusion map and also the homomorphism $SU(2)\rightarrow SO(3)\subseteq SU(3)$.  If the action of the $SU(2)$ factor of $H$ on the $SU(2)$ factor of $G$ is trivial, then these biquotients have the form $(SU(3)/SU(2))\times_{S^1} SU(2) = S^5\times_{S^1} S^3$ or $(SU(3)/SO(3))\times_{S^1} SU(2)$.

But the $SU(2)$ factor of $H$ can also act by conjugation on the $SU(2)$ factor of $G$.  When this occurs, according to Proposition \ref{su2biquotients}, the projection of the $H$ action on the $SU(3)$ factor of $G$ must be effectively free.  But then $SU(3)\bq H$ is a $4$-dimensional biquotient.  Such biquotients have been classified \cite{Es2,D1,KZ}.  It turns out that there are precisely two such actions:  for $(A,z)\in SU(2)\times S^1$ and $B\in SU(3)\times SU(2)$, they are  $$(A,z)\ast B = \diag(zA, \overline{z}^2) B \text{ and } (A,z)\ast B = \diag(zA, \overline{z}^2) B \diag(z^4, z^4, \overline{z}^8)^{-1}.$$  It follows that in either case, this biquotient is decomposable, being the total space of an $S^3$-bundle over $\mathbb{C}P^2$.  Using a calculation similar to the proof of Proposition \ref{calc1}, one can show that these two biquotients have cohomology rings isomorphic to that of $S^3\times \mathbb{C}P^2$, but that their first Pontryagin class distinguishes them.  In the case of $S^3\times \mathbb{C}P^2$, $p_1 = \pm 3\in H^4 \cong \mathbb{Z}$, and for the first biquotient, one has $p_1 = 0$, while for the second, $p_1 = \pm 8$.  Since $p_1 \pmod{24}$ is a homotopy invariant \cite{AH}, these three biquotients are distinct up to homotopy.

\

We now classify effectively free biquotient actions of $S^1$ on $(SU(3)/SO(3))\times S^3$.

\begin{proposition}

Consider the action of $H = SO(3)\times S^1$ on $G = SU(3)\times SU(2)$ given by \begin{align*} (A,z)\ast (B,C) = &\big( A B \diag(z^{m_1}, z^{m_2}, \overline{z}^{m_1+m_2}), \\ & \diag(z^{n_1}, \overline{z}^{n_1}) C \diag(z^{n_2}, \overline{z}^{n_2})^{-1}\big),\end{align*} with $\gcd(m_1, m_2, n_1, n_2) = 1$.  Then this action is free iff $\gcd(m_1 m_2 (m_1 + m_2), n_1^2 -  n_2^2) = 1$.  Further,  the action is effectively free iff it is free.

\end{proposition}

\begin{proof} We note that $\gcd(m_1 m_2(m_1 + m_2), n_1^2 - n_2^2) = 1$ iff $\gcd(m_i, (n_1 \pm n_2)) = 1 = \gcd(m_1 + m_2, n_1 \pm n_2)$.  It is easy to see that an action satisfying each of these six $\gcd$ conditions is free.

Let $f=(f_1,f_2):H\rightarrow G^2$ define the action and assume the action is effectively free.  According to Proposition \ref{freetest}, we see that whenever $f_1(A,z)$ is conjugate to $f_2(A,z)$, then $f_1(A,z) = f_2(A,z) \in Z(G)$.  Since $SO(3)\subseteq SU(3)$ does not intersect the center of $SU(3)$, if $f_1(A,z) \in Z(G)$, then $f_1(A,z)\in \{I\}\times Z(SU(2)) = \{(I,\pm I)\}.$  Since $f$ is injective and since $(I,-I)^2 = (I,I)$, the only elements of $H$ which can possibly map to $(I,-I)$ is $(I, -1)\in H$.  In particular, if the action by $SO(3)\times S^1$ on $G$ is ineffective, then $n_1$ and $n_2$ have the same parity.  We will later see that being effectively free implies $n_1$ and $n_2$ have different parities, so such an action is effectively free iff free.

In $SU(n)$, two matrices are conjugate iff they have the same eigenvalues, up to reordering.  In particular, the second factors are conjugate iff either $z^{n_1} = z^{n_2}$ or $z^{n_1} = \overline{z}^{n_2}$.  That is, the second factors are conjugate iff $z$ is either an $(n_1-n_2)$-th root of $1$ or an $(n_1+n_2)$-th root of $1$.

The eigenvalues of a matrix in $SO(3)$ have the form $\lambda, \overline{\lambda}, 1$.  It follows that a matrix of the form $\diag(z^{m_1}, z^{m_2}, \overline{z}^{m_1 + m_2})$ is conjugate to an element in $SO(3)$ iff $z$ is an $m_i$-th root of $1$, or an $(m_1 + m_2)$-th root of $1$.

Assume $z$ is a $\gcd((n_1-n_2), m_1)$-th root of $1$.  Then $(A,z)$ fixes a point of $SU(3)\times SU(2)$ for an appropriate choice of $A$, and thus, $(A,z)$ fixes every point of $SU(3)\times SU(2)$.  This implies $z^{n_1} = z^{n_2} \in \{1,2\}$ and $z^{m_1} = z^{m_2}=1$, from which is follows that $\overline{z}^{m_1 + m_2} = 1$.  In particular, $\gcd((n_1-n_2), m_1)$ must divide $\gcd(m_1,m_2, 2n_1, 2n_2)\in \{1,2\}$.

Analogously, each of $\gcd((n_1\pm n_2), m_i)$ and $\gcd((n_1\pm n_2), m_1 + m_2)$ must either be $1$ or $2$.  However, it actually follows that none of them can be $2$.  For if one of them is $2$, then $n_1$ and $n_2$ have the same parity and $m_1$ and $m_2$ are even. Since $\gcd(m_1,m_2, n_1, n_2) = 1$, $n_1$ and $n_2$ must both be odd.  But then one of $n_1 + n_2$ and $n_1 - n_2$, say $n_1 + n_2$, is divisible by $4$.  In addition, since both $m_i$ are even, at least one of $m_1$, $m_2$, and $m_1 + m_2$, say $m_1 + m_2$ is divisible by $4$.  But then $4\leq \gcd(n_1 + n_2, m_1 + m_2)\leq 2$, a contradiction.

We have now shown that if the action is effectively free, then $\gcd(m_i, n_1 \pm n_2) = 1$ and $\gcd(m_1 + m_2, n_1 \pm n_2) = 1$.  But this is equivalent to $\gcd(m_1 m_2(m_1 + m_2), n_1^2 - n_2^2) = 1$.  Because at least one of $m_1$, $m_2$, and $m_1 + m_2$ is even, $n_1$ and $n_2$ must have opposite parities, which, by the above discussion, implies the action is free.

\end{proof}

When $\{n_1,n_2\} = \{0,1\}$, we note that the biquotients are diffeomorphic to the product $(SU(3)/SO(3))\times S^2$.  To see this, first equip $SU(3)$ with a bi-invariant metric, so $SU(3)/SO(3$ is a symmetric space and the induced $S^1$ action it is isometric.  The isometry group is $SU(3)/(Z(SU(3))\cap SO(3)) = SU(3)$.   Now, $\pi_2(BSU(3)) = 0$ so the trivial bundle is the only principal $SU(3)$ bundle over $S^2$.  Thus, as in the discussion following Proposition \ref{cp2s2actions}, $(SU(3)/SO(3))\times_{S^1} S^3 \cong (SU(3)/SO(3))\times_{SU(3)}(SU(3)\times_{S^1} S^3) \cong (SU(3)/SO(3))\times_{SU(3)} (SU(3)\times S^2) \cong (SU(3)/SO(3))\times S^2$.

\

We now study linear actions of $S^1$ on $S^5\times S^3$ in more detail.  We identify $S^5\times S^3 =\{ (a,b)\in\mathbb{C}^3 \times \mathbb{C}^2: |a|^2 = |b|^2 = 1\}$.  Then, Proposition \ref{convert}, specialized to the case where $w=1$ shows that a homomorphism $SU(2)\times S^1\rightarrow SU(3)^2$ with $$(A,z)\mapsto \big( \diag(z^{\overline{m}_1} A, z^{2m_1}), \diag(z^{m_2}, z^{m_3}, \overline{z}^{m_1+m_2})\big)$$ induces the $S^1$ action on $S^5\times S^3$ $$z\ast(a_1,a_2,a_3) = (z^{2m_1-m_2} a_1, z^{2m_1-m_3} a_2, z^{2m_1 + m_2 + m_3} a_3)$$ and that a homomorphism $Sp(2)\times S^1\rightarrow SU(4)$ with $$(A,z)\mapsto\big( A, \diag(z^{m_1}, z^{m_2}, z^{m_3}, \overline{z}^{m_1+m_2 + m_3}) \big)$$ induces $$z\ast( a_1,a_2,a_3) = (z^{m_2 + m_3} a_1, z^{m_1+ m_3} a_2, z^{m_1 + m_2} a_3).$$

\begin{proposition}Suppose $S^1$ acts on $S^5\times S^3$ as $$z\ast(a_1,a_2,a_3, b_1,b_2) = (z^{m_1} a_1, z^{m_2} a_2, z^{m_3} a_3, z^{n_1} b_1, z^{n_2} b_2)$$ with $\gcd(m_1,m_2,m_3,n_1,n_2)  = 1$.  Then the action is free iff $\gcd(m_1 m_2 m_3, n_1 n_2) = 1$.  In addition, the action is free iff it is effectively free.

\end{proposition}

The proof of this proposition is very similar to that of Proposition \ref{cp2s2actions}, so is omitted.

If $|n_1| = |n_2| = 1$, then these biquotients are decomposable; they are naturally the total space of a linear $S^5$-bundle over $S^2$.  As in Section \ref{s2s4ands3s4}, such bundles are in bijection with $\pi_1(SO(6))\cong \mathbb{Z}_2$ and the total spaces are distinct up to homotopy.

Similarly, if $|m_1| = |m_2| = |m_3| = 1$, then these biquotients are decomposable; they are naturally the total space of a linear $S^3$ bundle over $\mathbb{C}P^2$.

We have now classified all effectively free actions in the case where $G\bq H$ has rational homotopy groups isomorphic to those of $S^3\times \mathbb{C}P^2$.  We summarize this in the following proposition.

\begin{proposition}

Suppose $G\bq H$ is a reduced biquotient having rational homotopy groups isomorphic to those of $S^3\times \mathbb{C}P^2$.  Then, up to equivalence, one of the following occurs.

(1)  $(G,H) = (SU(3), S^1)$ and $G\bq H$ is an Eschenburg space.

(2)  $(G,H) = (SU(4)\times SU(2), Sp(2)\times S^1)$ and the action is of the form $(A,z)\ast(B,C) =$ $$\big( AB\diag(z^{m_1}, z^{m_2}, z^{m_3}, \overline{z}^{m_1 + m_2 + m_3}), \diag(z^{n_1}, \overline{z}^{n_1}) C \diag(z^{n_2}, \overline{z}^{n_2})^{-1}\big)$$ with $\gcd(m_1,m_2, m_3, n_1, n_2) = 1$ and $\gcd((m_1+m_2)(m_1 + m_3)(m_2+m_3), n_1^2 - n_2^2) = 1$.

(3)  $(G,H) = (SU(3)\times SU(2), SU(2)\times S^1)$ and the action is one of the following four forms. 

(3A):  $(A,z)\ast(B,C) =$ $$\big( \diag(\overline{z}^{m_1} A, z^{2m_1}) B \diag(z^{m_2}, z^{m_3}, \overline{z}^{m_2 + m_3})^{-1}, \diag(z^{n_1}, \overline{z}^{n_1}) C \diag(z^{n_2}, \overline{z}^{n_2})^{-1}\big)$$ with $\gcd(m_1, m_2, m_3, n_1, n_2) = 1$ and $\gcd((2m_1  - m_2)(2m_1-m_3)(2m_1 + m_2 + m_3), n_1^2 - n_2^2)=1$.

(3B):  $(A,z)\ast(B,C) =$ $$\big( \pi(A) B \diag(z^{m_1}, z^{m_2}, \overline{z}^{m_1 + m_2}), \diag(z^{n_1}, \overline{z}^{n_1}) C \diag(z^{n_2}, \overline{z}^{n_2})^{-1} \big)$$ with $\gcd(m_1, m_2, n_1, n_2) = 1$ and $\gcd(m_1 m_2(m_1 + m_2), n_1^2 - n_2^2) = 1$.

(3C):  $(A,z)\ast(B,C) = \big( \diag(zA, \overline{z}^2) B , A CA^{-1}\big)$

(3D):  $(A,z)\ast(B,C) = \big( \diag(zA, \overline{z}^2) B \diag(z^4, z^4, \overline{z}^8)^{-1}, ACA^{-1} \big)$

\end{proposition}

\

We now focus on the topology of these examples, starting with the case of biquotients of the form $S^5\times_{S^1} S^3$.  The topology of these examples has been studied extensively.  See for example, \cite{Esc,Kru2,Kru3} for a partial diffeomorphism classification and \cite{Kru1} for a homotopy classification.

\begin{proposition}
If $n_1 n_2 = 0$, then the cohomology ring of $G\bq H = S^5\times_{S^1} S^3$ is isomorphic to the cohomology ring of $\mathbb{C}P^2 \times S^3$.

If $n_1 n_2 \neq 0$, then $$H^\ast( G\bq H) \cong \begin{cases} \mathbb{Z} & \ast=0,2,5,7\\ \mathbb{Z}_{n_1 n_2} & \ast = 4\\ 0 & \text{otherwise}\end{cases}$$ where the square of the generator of $H^2(G\bq H)$ generates $H^4(G\bq H)$.

Further, if $u\in H^2$ is a generator, then $p_1(G\bq H) = (m_1^2 + m_2^2 + m_3^2 - n_1^2 - n_2^2) u^2$, 
$w_2(G\bq H) = (\sigma_1(m_i) + \sigma_1(n_i)) u$, and $w_4(G\bq H) = (\sigma_2(m_i) + \sigma_1(m_i))\sigma_1(n_i) u^2$.
\end{proposition}

Kruggel \cite{Kru1} proves this in the case $n_1 n_2\neq 0$.  When $n_1 n_2 = 0$, the equation $\gcd(m_1 m_2 m_3, n_1 n_2) = 1$ forces, up to equivalence, all $m_i$ to be $1$.  One can then use the method of Section \ref{linearsphere} to see that Kruggel's formula for the first Pontryagin class and Stiefel-Whitney classes remains valid in this exceptional case.

For biquotients of the form $(SU(3)/SO(3))\times_{S^1} SU(2)$, one has the following result.

\begin{proposition}

For biquotients of the form $(SU(3)/SO(3))\times_{S^1} SU(2)$, the cohomology ring is $H^\ast(G\bq H) \cong \mathbb{Z}[u_2, u_3, u_5]/I$ where $|u_i| = i$ and $$I = \langle \gcd(n_1^2- n_2^2, m_1m_2 - (m_1+m_2)^2) u_2^2, u_2^3, 2u_3, u_3^2, u_5^2 \rangle.$$

Further, $p_1(G\bq H) = 4n_1^2 u_2^2 \in H^4(G\bq H)$.

With coefficients in $\mathbb{Z}_2$, one has $H^\ast(G\bq H;\mathbb{Z}_2) \cong \mathbb{Z}_2 [u_2, v_2, u_3]/J$ where $$J  =\langle u_2^2, v_2^2, u_3^2\rangle.$$ The Stiefel-Whitney class is $w(G\bq H) = 1 + u_2 + u_3$.
\end{proposition}

This proposition can be proved using the techniques from Section \ref{top} with one small change when using $\mathbb{Z}$ coefficients.  Namely, in this case $H^\ast(BSO(3))$ is not a polynomial algebra.  However, one can prove directly that, the map $H^4(BSO(3))\rightarrow H^4(BT_{SO(3)})$ induced from the inclusion of a maximal torus $T_{SO(3)}$ of $SO(3)$ is an isomorphism.  This allows one to compute the cohomology ring up to degree $4$ as well as the first Pontryagin class.  Poincar\'e duality then determines the rest of the cohomology ring.

\subsection{\texorpdfstring{Biquotients with $\pi_\ast(G\bq H)_\Q\cong \pi_\ast(S^3\times S^2\times S^2)_\Q$}{Rational homotopy groups like S3 x S3 x S2  }}\label{analysiss2s3s3}

In this section, we classify reduced biquotients with rational homotopy groups isomorphic to those of $S^3\times S^2\times S^2$.

According to Table \ref{table:gplist}, if $G\bq H$ is a reduced biquotient with rational homotopy groups isomorphic to those of $(S^3)\times (S^2)^2$, then $(G,H) = (SU(2)^3, T^2)$.  As in Sections \ref{s2s4ands3s4} and \ref{analysiss2cp2}, classifying such biquotient actions is equivalent to classifying linear $T^2$ actions on $(S^3)^3$.  As usual, we identify $S^1$ with the unit complex numbers and $S^3 = \{(a_1,a_2)\in \mathbb{C}^2: |a_1|^2 + |a_2|^2 = 1\}$.

A general linear $T^2$ action on $(S^3)^3$ takes the form $(z,w)\ast(a_1,a_2,b_1,b_2, c_1,c_2) =$ $$ (z^{k_1}w^{l_1} a_1, z^{m_1} w^{n_1} a_2, z^{k_2} w^{l_2} b_1, z^{m_2} w^{n_2} b_2, z^{k_3} w^{l_3} c_1 z^{m_3} w^{n_3} c_2).$$  We may assume that $\gcd(k_1, k_2, k_3, m_1, m_2, m_3) = \gcd(l_1, l_2, l_3, n_1, n_2, n_3) = 1$.

\begin{proposition} An action as above is free iff for every choice of three elements $(s_i,t_i)\in \{(k_i, l_i), (m_i, n_i)\}$ with $i=1,2,3$, we have $$\gcd \left(  s_1 t_2 - s_2 t_1, s_1 t_3 - s_3 t_1, s_2 t_3 - s_3 t_2 \right) = 1.$$

\end{proposition}

\begin{proof}  As discussed previously, a linear torus action on a product of spheres is free iff every point with all coordinates either $0$ or $1$ is moved by every non-trivial element of the torus.

Then, for example, at the point $(1,0,1,0,1,0)\in (S^3)^3$, we see that $(z,w)\in T^2$ fixes this point iff  $z^{k_1} w^{l_1} = z^{k_2}w^{l_2} = z^{k_3} w^{l_3} = 1.$

Let $X= \begin{bmatrix} k_1 & l_1 \\ k_2 & l_2\\ k_3 & l_3\end{bmatrix}$.  Then it is easy to see that every non-trivial element of $T^2$ moves the point $(1,0,1,0,1,0)$ iff $X^{-1}(\mathbb{Z}^3)\subseteq \mathbb{Z}^2$.  Indeed, $(z,w) = (e^{2\pi i v_1}, e^{2\pi i v_2})$ fixes (1,0,1,0,1,0) iff $(v_1 \ v_2)^t\in X^{-1}(\mathbb{Z}^3)$.  But $X^{-1}(\mathbb{Z}^3) \subseteq \mathbb{Z}^2$ iff $Y^{-1}(\mathbb{Z}^3)\subseteq \mathbb{Z}^2$, where $Y$ denotes the Smith normal form of $X$.

Set $\alpha = \gcd(k_2l_1 - k_1 l_2, k_3 l_1 - l_3k_1,k_1 l_3 - k_3 l_1)$.  Then the $Y = \begin{bmatrix} \beta & 0 \\ 0 & \alpha \\ 0 & 0\end{bmatrix}$ for some integer $\beta$ which divides $\alpha$.  It follows that there is a non-integral rational vector $v = (v_1, v_2)\in \mathbb{Q}^2$ with $X\cdot v \in \mathbb{Z}^3$ iff $\alpha \neq  1$.  Repeating this argument for the others points in $(S^3)^3$ with all coordinates $1$ or $0$, we see the condition on $\gcd$s is necessary and sufficient for the action to be free.

\end{proof}

We conclude with results on the cohomology groups and characteristic classes of these biquotients which are provable using the method found in Section \ref{linearsphere}.  Set $d = \det \begin{bmatrix} k_1 m_1 & k_1 n_1 + l_1 m_1 & n_1 l_1\\ k_2 m_2 & k_2 n_2 + l_2 m_2 & n_2 l_2 \\ k_3 m_3 & k_3 n_3 + l_3 m_3 & n_3 l_3\end{bmatrix}.$

\begin{proposition}  Consider a free linear $T^2$ action on $(S^3)^3$ paramaterized by integers $(k_i, l_i, m_i,n_i)$ as above.

If $d = 0$, then $H^\ast((S^3)^3/T^2)\cong \mathbb{Z}[t,u,v]/I$ where $|t|=|u|=2$ and $|v| = 3$ and $$I = \langle v^2, (k_i t + l_iu)(m_i t + n_i u) \rangle.$$

If $d\neq 0$, then $H^\ast((S^3)^3/T^2) \cong \mathbb{Z}[t,u,w,x]/J$ where $|t| = |u| =2$ and $|w| = |x| =5$ and $$J = \langle w^2, (k_i t + l_iu)(m_i t + n_i u) \rangle.$$

In either case, $p_1((S^3)^3/T^2) = \sum_{i=1}^3 (k_i-m_i) t^2 + (l_i - n_i) u^2$, $w_2((S^3)^3/T^2) = \sum_{i=1}^3 (k_i + m_i)t + (l_i + n_i)u$, and $w_4((S^3)^3/T^2) = \sum_{1\leq i < j\leq 3} ((k_i + m_i)t + (l_i + n_i)u)((k_j + m_j)t + (l_j + n_j)u).$

\end{proposition}

%


%

\end{document}